\documentclass[a4paper,12pt,reqno]{mymathart}

\usepackage{graphicx} 
\usepackage{psfrag}

\usepackage[pdftex]{thumbpdf} 

\usepackage[latin1]{inputenc} 
\usepackage[T1]{fontenc}
\usepackage{float}

\usepackage{enumerate}
\usepackage{diagrams}

\usepackage[reqno]{amsmath}     
\usepackage{amssymb}            
\usepackage{mathrsfs}           
\usepackage{dsfont}             
\usepackage{amsxtra} 

\newtheorem{Theorem}{Theorem}[section]

\newtheorem{Definition}[Theorem]{Definition}
\newtheorem{dfn&prop}[Theorem]{Definition and Proposition}
\newtheorem{Proposition}[Theorem]{Proposition}
\newtheorem{Corollary}[Theorem]{Corollary}

\theoremstyle{remark}

\newtheorem*{Remark}{Remark}
\newtheorem*{Claim}{Claim}
\newtheorem*{proof of claim}{Proof of claim}

\theoremstyle{definition}

\newtheorem*{notations}{Notations}

\numberwithin{equation}{section}

\def\F{{\mathcal{F}}}
\def\J{{\mathcal{J}}}

\def\M{{\mathcal{M}}}
\def\Stp{{\mathscr{S}}}
\def\Exad{{\mathscr{S}^{\mathbb{N}}}}

\def\C{{\mathbb{C}}}
\def\D{{\mathbb{D}}}
\def\N{{\mathbb{N}}}
\def\Z{{\mathbb{Z}}}
\def\R{{\mathbb{R}}}
\def\H{{\mathbb{H}}}

\def\Or{{\mathcal{O}}}
\def\Ort{{\widetilde{\mathcal{O}}}}

\newcommand{\dist}{\operatorname{dist}}

\newcommand{\e}{\operatorname{e}}
\newcommand{\degr}{\operatorname{deg}}

\newcommand{\lcm}{\operatorname{lcm}}

\newcommand{\Oo}{\operatorname{O}}
\newcommand{\id}{\operatorname{id}}
\newcommand{\len}{\operatorname{\alpha}}
\newcommand{\Ima}{\operatorname{Im}}
\newcommand{\Rea}{\operatorname{Re}}

\newcommand{\Crit}{\operatorname{Crit}}

\newcommand{\Stab}{\operatorname{Stab}}
\newcommand{\itin}{\operatorname{itin}}
\newcommand{\ul}{\underline}
\newcommand{\cl}{\overline}

\usepackage{hyperref}
\title[Semiconjugacies and pinched Cantor bouquets]{Semiconjugacies, pinched Cantor bouquets and hyperbolic orbifolds}
\author{Helena Mihaljevi\'{c}-Brandt}

\begin{document}

\maketitle

\begin{abstract}
Let $f:\C\rightarrow\C$ be a transcendental entire map that is \emph{subhyperbolic}, 
i.e.,\ the intersection of the Fatou set $\F(f)$ and the postsingular set $P(f)$ 
is compact and the intersection of the Julia set $\J(f)$ and $P(f)$ is finite. 
Assume that no asymptotic value of $f$ belongs to $\J(f)$ 
and that the local degree of $f$ at all points in $\J(f)$ is bounded 
by some finite constant. We prove that there is a hyperbolic map 
$g\in\{z\mapsto f(\lambda z):\; \lambda\in\C\}$ with 
connected Fatou set such that $f$ and $g$ are semiconjugate on their Julia sets. 
Furthermore, we show that this semiconjugacy is a conjugacy
when restricted to the escaping set $I(g)$ of $g$. 
In the case where $f$ can be written as a finite composition of maps of finite order, 
our theorem, together with recent results on Julia sets of hyperbolic maps,
implies that $\J(f)$ is a \emph{pinched Cantor bouquet}, consisting of dynamic rays 
and their endpoints. Our result also seems to give the first complete description 
of topological dynamics of an entire transcendental map whose Julia set is the whole 
complex plane.
\end{abstract}
\section{Introduction}
It is well-known that the Julia set $\J(f)$ of a 
 transcendental entire function $f$ can be the whole complex plane.
 (Basic definitions and notations 
  are reviewed in Section \ref{sec_prel}.)
 As far as we know, there is no function with this property for which
 the topological dynamics has been completely understood. The results in
 this paper provide such a description for a wide class of examples, 
 including maps such as $z\mapsto \pi \sinh z$. (A 
 description of the \emph{combinatorial} dynamics of the latter map was
 previously given by Schleicher in \cite{schleicher}.)

In particular, we give an answer to the question of Bergweiler (personal  
 communication) whether the escaping set
 \begin{align*}
  I(f):=\{ z\in\C: f^n(z)\to\infty\text{ as }n\to\infty\}
 \end{align*}
 of a cosine map $F_{a,b}(z):=a\e^z +b\e^{-z}$ 
with strictly preperiodic critical values is connected: this is not the case;
 see Corollary \ref{cor_bergweiler} below.
 (On the other hand, there are entire 
  maps whose Julia set equals $\C$ but 
  for which the escaping set is connected; see
  \cite{rippon_stallard,rempe_6}.) 

 In fact, our results are considerably more general. 
  A transcendental entire map $f$ is called \emph{subhyperbolic} if 
  the intersection of the Fatou set $\F(f)$ and the postsingular set $P(f)$ 
  is compact and the intersection of the Julia set $\J(f)$ and $P(f)$ is finite.
  A subhyperbolic map is called \emph{hyperbolic} 
  if $\J(f)\cap P(f)=\emptyset$. 
  We are interested in the following class of subhyperbolic functions which
  includes all hyperbolic maps.

\begin{Definition}[strongly subhyperbolic maps]
 A subhyperbolic transcendental entire map $f$ is called 
  \emph{strongly subhyperbolic} if  
  $\J(f)$ contains no asymptotic values of $f$  
  and the local degree of $f$ at the points in $\J(f)$ is bounded by some 
  finite constant. 
\end{Definition} 

 Note that the map $z\mapsto \pi\sinh z$ mentioned previously is 
  strongly subhyperbolic, since $P(f)=P(f)\cap \J(f)=\{\pm \pi i, 0\}$ is
  finite, all critical points are simple and there are no asymptotic values. 
  Our main theorem describes the Julia set of any strongly subhyperbolic
  entire function as a quotient of the Julia set of 
  a (particularly simple) hyperbolic function in the same parameter space. 

\begin{Theorem}
\label{thm_maintheorem}
Let $f$ be strongly subhyperbolic, and let
 $\lambda\in\C$ be such that $g(z):=f(\lambda z)$ is hyperbolic with connected
 Fatou set. Then there exists 
 a continuous surjection $\phi:\J(g)\rightarrow\J(f)$, such that 
\begin{eqnarray*}
f(\phi (z)) = \phi( g(z))
\end{eqnarray*}
for all $z\in\J(g)$. 
Moreover, $\phi$ restricts to a homeomorphism between the escaping sets
 $I(g)$ and $I(f)$.
\end{Theorem} 
\begin{Remark}
 The hypothesis will be automatically satisfied whenever $\lambda$ is
  sufficiently small. Also, any two maps $g$ and $g'$ as in the theorem
  are quasiconformally conjugate on a neighbourhood of their Julia sets
  \cite{rempe_5}, so it is sufficient to prove the theorem for 
  any such map. 
\end{Remark}

As in \cite{rempe_5}, 
 we say that a hyperbolic function $g$ with connected Fatou set 
 is of \emph{disjoint type}. 
 For simple families, such as $z\mapsto \lambda \sinh z$, the
 dynamics of disjoint type functions is well-understood. Hence Theorem
 \ref{thm_maintheorem} extends this understanding to all strongly 
 subhyperbolic functions in these families. 
(In Appendix A, we present a detailed description of the
 topological dynamics of $z\mapsto\pi\sinh$.) 

More generally, suppose that $g$ is of disjoint type and
 has \emph{finite order}, i.e., 
$\log\log\vert g(z)\vert =\Oo(\log\vert z\vert)$ as 
$z\to\infty$, or, more generally, 
can be written as a finite composition of finite-order maps with bounded 
singular sets. Then it is known 
that $\J(g)$ 
is a \emph{Cantor bouquet}, i.e.,\ homeomorphic to a \emph{straight brush} in
the sense of \cite{aarts_oversteegen}. 

\begin{Corollary}
\label{cor1}
Let $f=f_1\circ\dots\circ f_n$ be a strongly subhyperbolic map, where 
every $f_i$ is an entire map with bounded set of singular values and with finite order 
of growth. 
Then $\J(f)$ is a \emph{pinched Cantor bouquet}; that is, the quotient
 of a Cantor Bouquet by a closed equivalence relation defined on its endpoints. 
\end{Corollary}

The escaping set of a disjoint type entire function is always disconnected,
 so we also have the following corollary, settling
 Bergweiler's question for all strongly subhyperbolic maps. 
\begin{Corollary}
\label{cor_bergweiler}
 The escaping set of a strongly subhyperbolic transcendental 
  entire function is disconnected.
\end{Corollary}

Let us comment on the assumption of strong subhyperbolicity.
 Theorem \ref{thm_maintheorem} is not true for all subhyperbolic maps: 
 it is known that $E_1:z\mapsto \frac{1}{e^2} e^z$ (disjoint type) and
 $E_2:z\mapsto  2\pi i e^z$ (subhyperbolic)
 are not topologically conjugate on their escaping sets
 \cite[Proposition 2.1]{rempe_1}. 
(In fact, for exponential
 maps the escaping set consists of curves to infinity, called
 \emph{dynamic rays} \cite{schleicher_zimmer}. For $E_1$, all these
 rays have a landing point in $\C$, while for $E_2$ there are uncountably
 many dynamic rays, each of which accumulates everywhere upon itself.) 
Rempe asked whether two 
cosine maps with strictly preperiodic critical values 
can be conjugate on their escaping sets
\cite[Question 12.1]{rempe_1}.
Theorem \ref{thm_maintheorem} together with the mentioned result in \cite{rempe_5}
on disjoint type maps gives an affirmative answer to this question.

For cosine maps $F_{a,b}(z)= a\e^z + b\e^{-z}$ that are subhyperbolic, 
Theorem \ref{thm_maintheorem} implies that every point in the Julia set 
is either on a dynamic ray or the landing point of a dynamic ray. 
This has already been shown by 
Schleicher \cite{schleicher} when both critical values of $F_{a,b}$
are assumed to be preperiodic. Nonetheless, his results do not explain 
the topological embedding of the escaping set of such a map in the complex plane.
Furthermore, 
Theorem \ref{thm_maintheorem} formulated for subhyperbolic cosine maps can be proved in a 
concise and fairly elementary way,
which is why we have included the modifications of our proof for this special case 
in Section \ref{subs_cosine}.

For hyperbolic maps, Theorem \ref{thm_maintheorem} is 
due to Rempe, and our proof 
is in the spirit of the ideas presented in \cite{rempe_5}.
However, the attempt to transfer the construction in the hyperbolic 
case to the setting of strongly subhyperbolic maps fails 
due to the existence of singular values in the Julia sets.
This obstruction is overcome by studying Julia sets as subsets of 
\emph{hyperbolic Riemann orbifolds}. These can be thought of as 
images of the unit disk under branched coverings, for which the 
set of critical values is discrete and ``tame''. 
This yields a description of a Riemann orbifold 
as a Riemann surface together with a
discrete set of \emph{ramified points}, each of which has 
finite ramification value.
The use of orbifolds in dynamics goes back to Thurston 
and has been used with great success by Douady and Hubbard in their work
on subhyperbolic rational maps.
The following result 
is not only crucial for the proof of Theorem \ref{thm_maintheorem} but also 
interesting in its own right, since it provides us with a global estimate 
of the hyperbolic metric on certain hyperbolic Riemann orbifolds.

\begin{Theorem}
\label{thm_main2}
Let $K>1$ and let $z_i$ be an infinite sequence of points satisfying 
$\vert z_j\vert <\vert z_{j+1}\vert\leq K\vert z_j\vert$.
Let $\Or$ be the orbifold with $\C$ as the underlying surface and 
whose ramified points are the points
$z_i$ with ramification value $2$. 

Then the density $\rho_{\Or}$ of the hyperbolic metric  
on $\Or$ satisfies
\begin{align*}
\rho_{\Or}(z)\geq\Oo\left(\frac{1}{\vert z\vert}\right)\quad\text{as}\quad z\rightarrow\infty.
\end{align*}
\end{Theorem}

The orbifold for which the set of ramified points
is given by $\{2 k\pi i: k\in\Z\}$  shows that our 
estimate is best possible (see proof of Proposition \ref{prop_cos_est}).
\begin{Remark}
If we replace the ramified points by punctures, i.e., if we consider 
the hyperbolic domain $U:\C\setminus\{ z_j\}$ instead of the orbifold $\Or$, 
the same bound for the asymptotic behaviour of the 
density map near $\infty$ can be obtained using standard estimates of 
the hyperbolic metric in the twice-punctured plane \cite[Lemma 2.1]{rempe_5}.
\end{Remark}

 It will become clear in Section \ref{sec associate} 
why strongly subhyperbolic maps are exactly those maps 
which can be approached with orbifold theory. 
This still leaves open
the question what can be said about topological dynamics of maps that 
are subhyperbolic but not strongly subhyperbolic. 
We believe that if $\J(f)$ contains an asymptotic 
value of $f$ then, similar to the case of exponential maps, there is no 
conjugacy between $f$ and any (suitable) disjoint type map $g$  
on their escaping sets; this is work in progress. 
However, we have no indication of what to expect for maps 
whose Julia sets contain no asymptotic values but sequences of points with 
unbounded local degree. It would be very 
interesting to explore this problem, in particular since there 
are prominent examples of such maps like \emph{Poincar\'{e} functions} 
corresponding to certain hyperbolic polynomials; an elaboration 
of such an example is given in Appendix B.

\subsection*{Structure of the article}
In Section \ref{sec associate} we 
develop the concept of orbifolds \emph{dynamically associated}
to a strongly subhyperbolic map $f$; the main consequence is that we obtain a
hyperbolic orbifold $\Or_f$ such that $f$ is expanding with respect to the  corresponding 
hyperbolic metric. 
Later, in Section \ref{sec uniform}, we prove that the expansion of $f$ is actually 
uniform. 
The key for this will be (the proof of) Theorem \ref{thm_main2}.
Finally, Section \ref{sec semiconjugacy} addresses 
the construction of the semiconjugacy itself.

\subsection*{Acknowledgements}
My special thanks go to my supervisor, Lasse Rempe for introducing me to 
the research problem and for his continuous help and support. Furthermore, 
I would like to thank Adam Epstein and Mary Rees for interesting and helpful 
discussions. 

\section{Preliminaries}
\label{sec_prel}

If not stated differently, we will assume throughout this article
that the considered maps are transcendental entire.
We denote the complex plane by $\C$, 
the Riemann sphere by $\widehat{\C}:=\C\cup\{\infty\}$, 
and the punctured plane by $\C^{*}:=\C\setminus\{ 0\}$.
We write $\D$ for the unit disk 
and $\H$ for the upper half-plane. 
The Euclidean disk centred at $c$ with radius $r$ is denoted by $D_r(c)$.
If not stated differently, the boundary $\partial A$ and the closure $\overline{A}$ 
of a set $A\subset\C$ is always understood to be taken relative to the complex plane. 
The Euclidean distance between two sets $A,B\subset\C$ will be denoted by $\dist(A,B)$.
For a sequence $n_i$ of natural numbers we write $\lcm\lbrace n_i\rbrace$ 
for their least common multiple.

\subsection{Background on holomorphic dynamics}
\label{subs_hol_dyn}
The \emph{Fatou set} of a map $f$ is the set of all points in $\C$ that have a 
neighbourhood in which the iterates $\{f^n\}$ form a normal family 
in the sense of Montel. Its complement $\J(f):=\C\setminus\F(f)$ is called the 
\emph{Julia set} of $f$. Recall that the \emph{escaping set} of $f$ is given by
\begin{eqnarray*}
I(f):=\lbrace z\in\C: \; f^{n}(z)\rightarrow\infty\text{ as }n\rightarrow\infty\rbrace.
\end{eqnarray*}
We say that a point $z\in\C$ 
is a \emph{periodic point} of $f$ if 
there exists an integer $n\geq 1$ such that $f^n(z)=z$. 
The smallest $n$ 
with this property is called the \emph{period} of $z$. A periodic 
point of period one is called a \emph{fixed point}. 
We call a point $z$ \emph{preperiodic} under $f$
if some image $f^n(z)$, $n\geq1$, of $z$ is periodic.
Note that every periodic point is also preperiodic.
To avoid confusion, we say that a point $z$ is \emph{strictly preperiodic} 
if it is preperiodic but not periodic. 
Let $z$ be a periodic point of $f$ of 
period $n$. 
 We call $\mu(z):=(f^n)^{'}(z)$ the 
\emph{multiplier} of $z$. A periodic point $z$ is called 
\emph{attracting} if $0\leq \vert \mu(z)\vert<1$, 
\emph{indifferent} if $\vert \mu(z)\vert=1$ and 
\emph{repelling} if $\vert \mu(z)\vert>1$. 
Since the multiplier of an indifferent periodic point 
is of the from $\e^{2\pi i t}$ with $0\leq t<1$, 
we can distinguish between \emph{rationally} and 
\emph{irrationally indifferent} points, according to 
whether $t$ is rational or not.
The set of all points whose orbits converge to an attracting periodic cycle
 is called the \emph{attracting basin} of this cycle. 

We denote the set of all critical points of $f$ by $\Crit(f)$, 
the set of all critical values by $C(f)=f(\Crit(f))$ 
and the set of all (finite) asymptotic values by $A(f)$. 
The set of \emph{singular values} of $f$, denoted by $S(f)$, 
is the smallest closed set such that 
$f: \C\setminus f^{-1}(S(f))\rightarrow \C\setminus S(f)$ is a covering map. 
It is well-known that $S(f):=\overline{C(f)\cup A(f)}$.
Finally, we denote the \emph{postsingular set} of $f$ by $P(f):=\overline{\bigcup_{n\geq 0} f^{n}(S(f))}$.

For more background on holomorphic dynamics we refer the reader to \cite{milnor,bergweiler1}.

\subsection{Background on Riemann orbifolds}
\label{background r o}

An \emph{orbifold} is a space which is locally modelled 
on the quotient of an open set in $\R^n$ by the 
linear action of a finite group. 
For a general introduction see \cite[$\S$ 13]{thurston_1}. 
Throughout this article we will need only  orbifolds modelled 
on Riemann surfaces, and for a more detailed introduction to
this topic see e.g. \cite{thurston_2, mcmullen_2, milnor}.

\begin{Definition}[Riemann orbifold]
\label{def r o}
A \emph{Riemann orbifold} is a pair $(S,\nu)$, 
where $S$ is a Riemann surface and 
$\nu:S\rightarrow\N_{\geq 1}$ is a map called the \emph{ramification map}, such that  
\begin{eqnarray*}
\lbrace z\in S\; :\; \nu(z)>1\rbrace
\end{eqnarray*}
is discrete. A point $z\in S$ with $\nu(z)>1$ is called a 
\emph{ramified} or \emph{marked point}. 
The \emph{signature} of an orbifold is the list of values 
that the ramification map $\nu$ assumes at the ramified points, 
where a value is repeated as often as it occurs as $\nu(z)$ for some ramified $z\in S$.
\end{Definition} 
A traditional Riemann surface can be regarded as a 
Riemann orbifold with ramification map $\nu\equiv 1$. 
Throughout this article, whenever we use the expression orbifold, we  
will always mean a Riemann orbifold. 

Recall that for a holomorphic map  $f:\widetilde{S}\rightarrow S$ 
between Riemann surfaces, the \emph{local degree} $\degr(f,z_0)$ of $f$ 
at a point $z_0\in \widetilde{S}$ is the unique integer $n=n(z_0)\geq 1$, 
such that 
\begin{eqnarray*}
f(z)=f(z_0) + a_n (z-z_0)^n + \text{(higher terms)}
\end{eqnarray*}
and $a_n\neq 0$. Thus $z_0$ is a \emph{critical} or 
\emph{branch point} if and only if $n(z_0)>1$.

The map $f$ is called a \emph{branched covering map} if every 
point in $S$ has a connected neighbourhood $U$ such that $f$ maps any 
component of $f^{-1}(U)$ onto $U$ as a proper map. 
Recall that a map $f:\widetilde{V}\rightarrow V$ is called \emph{proper} 
if the preimage $f^{-1}(K)$ of any compact set $K\subset V$ is a compact 
subset of $\widetilde{V}$. 

\begin{Definition}[Holomorphic map, covering]
Let $\Ort=(\widetilde{S},\tilde{\nu})$ and $\Or=(S, \nu)$ 
be Riemann orbifolds. A \emph{holomorphic map} $f:\Ort\rightarrow\Or$ is 
a holomorphic map $f:\widetilde{S}\rightarrow S$ between the underlying 
Riemann surfaces such that, for each $z\in\widetilde{S}$,
\begin{eqnarray}
\nu(f(z)) \text{ divides } \degr(f,z)\cdot \tilde{\nu}(z).
\end{eqnarray}
If $f:\widetilde{S}\rightarrow S$ is a branched covering map 
with $\nu(f(z)) = \degr(f,z)\cdot\tilde{\nu}(z)$ 
for all $z\in \widetilde{S}$, then $f:\Ort\rightarrow\Or$ 
is an \emph{orbifold covering map}. If additionally the surface $\widetilde{S}$ 
is simply-connected, then we call $\Ort$ a 
\emph{universal covering orbifold} of $\Or$.
\end{Definition}
\begin{Remark}
In the standard terminology, where an orbifold is defined via atlases
and group actions, the definition of a holomorphic map $f$ between 
two orbifolds is equivalent to a ``local lifting property'';
if $f$ is a covering then every such local lift can be chosen to
be an embedding. For more details, see \cite[A2]{mcmullen_2}.
\end{Remark}

Note that if $f:\Ort\to\Or$ is a covering then this is not necessarily 
true for the map $f:\widetilde{S}\to S$ between the underlying surfaces. 

Recall that by the Uniformization Theorem for Riemann surfaces, every 
Riemann surface has a universal cover that is conformally 
equivalent to either $\widehat{\C}$, $\C$ or $\D$. 
The following theorem tells us that the same is true 
for almost all Riemann orbifolds.
\begin{Theorem}[Uniformization of Riemann orbifolds]
\cite[Theorem A2]{mcmullen_2}
\label{uniform}
Let $\Or=(S,\nu)$ be a Riemann orbifold. Then $\Or$ has no universal 
covering orbifold if and only if $\Or$ is isomorphic to $\widehat{\C}$ 
with signature $(l)$ or $(l,k)$, where $l\neq k$. In all other cases the 
universal cover is unique up to conformal isomorphism over the surface $S$ and 
hence given by either $\widehat{\C}$, $\C$ or $\D$.
\end{Theorem}

In analogy to Riemann surfaces, 
we will call an orbifold $\Or$ \emph{elliptic, parabolic} or \emph{hyperbolic} 
if it is covered by $\widehat{\C}, \C$ or $\D$, respectively. 

\begin{Remark}
Let $\Ort,\Or$ be orbifolds that have a universal cover. 
Then a map $f:\Ort\to\Or$
is a covering if and only if 
it lifts to a conformal isomorphism 
between the universal covering spaces \cite[Lemma E.2]{milnor}.
\end{Remark}

For a connected orbifold $\Or=(S,\nu)$ the \emph{Euler characteristic} 
$\chi(\Or)$ is given by the equation
\begin{eqnarray*}
\chi(\Or) := \chi(S) - \sum_{z\in S} \left( 1 -\frac{1}{\nu(z)}\right),
\end{eqnarray*}
where $\chi(S)$ denotes the Euler characteristic of the surface $S$. 
Note that ramified point cause a reduction of $\chi(\Or)$. 
As for Riemann surfaces, a Riemann orbifold with negative 
Euler characteristic is always hyperbolic. This also implies that, 
roughly speaking, most orbifolds are hyperbolic. For the precise 
list of spherical and parabolic orbifolds, see the details  
of \cite[Theorem A2]{mcmullen_2}. 

Let $C$ be the uniformized universal covering surface of $\Or$ 
(i.e., $C\in\{\widehat{\C},\C,\D\}$) and denote by $\rho_C(z)\vert dz\vert$ 
its unique complete conformal metric of constant 
curvature $1$, $0$ or $-1$, respectively. By pushing forward this 
metric by a universal covering map 
we obtain a Riemannian metric on $\Or$ that can be written 
as $\rho_{\Or}(w)\vert dw\vert$ (in terms of a 
local uniformizing parameter $w$), and $\rho_{\Or}(w)$ is nonzero and smooth 
except at the ramified points of $\Or$. 
We call this metric the \emph{orbifold metric} of $\Or$.
Observe that at a ramified  point, say $w_0$, with ramification value $m$, 
the density has a singularity of the type
$\vert w-w_0\vert^{(1-m)/m}.$
More precisely, if we choose a local branched covering near $0$, e.g.,
$z(w)=(w-w_0)^m$, then the induced metric $\rho(z(w))\vert dz/dw\vert\cdot\vert dw\vert$
is smooth and nonsingular throughout some neighbourhood of $0$ in 
the $z$-plane.

Note that $\rho_{\Or}(w)\vert dw\vert$ is again a complete metric with  
constant curvature $1$, $0$ or $-1$, respectively, everywhere except at 
the marked points (which are singularities of the curvature).

In this article we are mainly interested in hyperbolic Riemann orbifolds.
The well-known Pick Theorem for hyperbolic surfaces generalizes to
hyperbolic orbifolds as well and will be of great use for us.
\begin{Theorem}
\cite[Proposition 17.4]{thurston_2}
\label{pick}
A holomorphic map between two hyperbolic orbifolds can never 
increase distances as measured in the hyperbolic orbifold metric. 
Distances are strictly decreased, unless the map is a covering map; 
in this case it is a local isometry.
\end{Theorem}
In particular, if $\Ort$ and $\Or$ are two orbifolds such that 
$\Ort\hookrightarrow\Or$ is holomorphic, then 
$\rho_{\Ort}(z)\geq\rho_{\Or}(z)$ for all $z\in\Ort$.

\subsection{Hyperbolic and subhyperbolic maps}
Recall that our main result is the existence of a semiconjugacy 
between a strongly subhyperbolic and a disjoint type map on their Julia sets. 
We will now recall the definitions and present briefly the relevant dynamical properties 
of those types of functions.
\begin{Definition}
\label{def_hyp}
A transcendental entire function $g$ is called \emph{hyperbolic} if $P(g)$ is a 
compact subset of $\F(g)$. 
\end{Definition}
It follows from classical results in holomorphic dynamics that 
the Fatou set of a hyperbolic (transcendental entire) 
map is non-empty and consists of attracting basins 
that correspond to finitely many attracting cycles. 

A hyperbolic \emph{rational} map is classically defined as a function 
which \emph{expands} a conformal Riemannian metric defined on a neighbourhood 
of its Julia set. 
For entire transcendental maps 
one can give a similar description: it is easy to derive from 
Definition \ref{def_hyp} that if $g$ is hyperbolic, then there exists a 
bounded neighbourhood $D$ of $P(g)$ such that $\overline{g(D)}\subset D$.
Note that by Montel's Theorem, $g^n\vert_D$ is a normal family, hence 
$\C\setminus\overline{D}$ is a neighbourhood of $\J(g)$. 
Furthermore, the map $g:\C\setminus g^{-1}(\overline{D})\to \C\setminus\overline{D}$ is
then a covering map which \emph{uniformly expands} the hyperbolic metric on the domain
$\C\setminus\overline{D}$ \cite[Lemma 5.1]{rempe_5}.

Note that not every hyperbolic rational map satisfies 
Definition \ref{def_hyp}:
By \cite[Theorem 19.1]{milnor}, a rational map is hyperbolic  
if and only if $P_{\J}=\emptyset$ or, equivalently, 
every critical point converges to an attracting periodic orbit.
In certain cases (including all nonconstant polynomials), the point at 
$\infty$ can be such an attractor as well.

\begin{Definition}
\label{disjoint type}
A hyperbolic transcendental entire map 
$g$ is said to be of \emph{disjoint type} if $\F(g)$ is connected.
\end{Definition}

Since every component of the Fatou set of a hyperbolic transcendental entire map is simply-connected 
\cite[Proposition 3, Theorem 1]{eremenko_lyubich_2}, it follows 
that the Fatou set of a disjoint type map is connected and simply-connected.

\begin{Proposition}
\label{prop_dt}
The following statements are equivalent:
\begin{enumerate}
\item The map $g$ is of disjoint type.
\item $g$ has a unique attracting fixed point and $P(g)$ is a 
compact subset of its immediate basin of attraction.
\item There exists a bounded Jordan domain $D\supset S(g)$ 
such that $\overline{g(D)}\subset D$.
\end{enumerate}

\end{Proposition}
\label{prop_disjoint_type}
\begin{proof}
Let $g$ be of disjoint type. In particular, $g$ is hyperbolic. Hence  
$P(g)$ is a compact subset of $\F(g)$. By definition, 
$\F(g)$ is connected, hence $P(g)$ is a compact subset of a 
completely invariant component of $\F(g)$, which can only be the immediate 
attracting basin of an attracting fixed point of $g$, 
showing that $(1)$ implies $(2)$.

We will just give a sketch of $(2)$ implies $(3)$,
for more details see e.g. proof of Proposition 3.7 in 
\cite{mihaljevic_brandt}. 
Let $z_0$ be the unique attracting fixed point of $g$
with immediate attracting basin $A^{*}(z_0)\supset P(g)$. 
Note that $P(g)$ must contain $z_0$. 
Furthermore, $P(g)$ has positive distance 
to $\partial A^{*}(z_0)$, hence there
is a simply-connected domain $D_0$ compactly
contained in $A^{*}(z_0)$ that contains a 
neighbourhood
of $P(g)$. Let $U$ be a linearising neighbourhood of $z_0$ and 
let $n$ be the smallest integer such that $g^n(\cl{D_0})\subset U$.
Taking the union of $\cl{U}$ and the compact sets 
$\cl{D_0}, g(\cl{D_0}),\dots , g^{n-1}(\cl{D_0})$ we obtain a compact connected
set $K$ that is mapped into its interior.
$K$ is not necessarily simply-connected; 
by the Open Mapping Theorem, we can fill the holes and obtain a
full set $\widetilde{K}$ that is mapped into its interior.
Now we choose $D$ to be a Jordan domain such that 
$g(\widetilde{K})\subset D\subset \widetilde{K}$.

To see that $(3)$ implies $(1)$, let us choose a 
domain $D\supset S(g)$ such that $\overline{g(D)}\subset D$. 
By Montel's Theorem, $D$ is contained in a component of $\F(g)$, 
so, in particular, $g$ is hyperbolic and $\F(g)$ is the union of attracting basins. 
Since every immediate attracting basin contains at least one 
singular value \cite[Theorem 7]{bergweiler1}, 
$g$ has a unique attracting cycle, which is a fixed point contained in $D$. 
Hence every point $z\in\F(g)$ is eventually mapped into $D$, showing that 
$\F(g)=\bigcup_{n\geq 0} g^{-n}(D)$. 
On the other hand, $g^{-(n+1)}(D)\supset g^{-n}(D)$ and $\F(g)$ is connected.
\end{proof}

Let $g$ be a transcendental entire map with bounded singular set $S(g)$,
let $D\subset\C$ be a bounded Jordan domain containing $S(g)$ 
and let $U:=\C\backslash\overline{D}$. Then $g :g^{-1}(U)\rightarrow U$ is a 
covering map and each component $T$ of $g^{-1}(U)$ is a simply 
connected unbounded Jordan domain; we call every such domain 
$T$ a \emph{tract} of $g$. 
Note that if $g$ is of disjoint type, then by Proposition 
\ref{prop_disjoint_type} we can choose $D$ such that 
$g^{-1}(\C\backslash \overline{D})$ is disjoint from $\overline{D}$. 
Using such a domain one can easily prove the following (well-known)
property of the escaping set of a disjoint type map.

\begin{Proposition}
\label{thm_esc_set_dis_type}
Let $g$ be of a disjoint type. Then $I(g)$ is disconnected.
\end{Proposition}
\begin{proof}
Let $D\supset S(f)$ be a Jordan domain such that 
$g^{-1}(\C\backslash \overline{D})$ is disjoint from $\overline{D}$
 and let $\alpha\subset f^{-1}(D)\setminus \overline{D}$
be a simple curve that connects $\partial D$ to $\infty$.
(Such a curve always exists since the boundary of the open set 
$f^{-1}(D)\setminus \overline{D}$ in $\widehat{\C}$ is locally connected.) 
The preimages of $\alpha$ split every tract of $g$ (w.r.t. $D$)
in simply-connected unbounded domains on which $g$ restricts as a conformal isomorphism;
we call every such domain a fundamental domain. Since $\cl{g(D)}\cup g(\alpha)\subset D$,
it follows that $I(g)$ is contained in the union of the
fundamental domains.

By \cite[Theorem ]{eremenko_1}, there exists a point $w\in I(g)$.
Hence every fundamental domain must intersect $I(g)$ since
it contains a preimage of $w$.
Let $\widetilde{F}$ be an arbitrary but fixed 
fundamental domain and let $U$ denote the union of 
all fundamental domains other than $\widetilde{F}$. 
Then $U$ and $\widetilde{F}$ are two disjoint 
nonempty open 
sets whose union covers $I(g)$. 
Thus $I(g)$ is disconnected.
\end{proof}

\begin{Definition}
\label{dfn_subh}
 A transcendental entire function $f$ is called \emph{subhyperbolic} if
 \begin{itemize}
\item[$(i)$]$P_{\J}:=P(f)\cap\J(f)$ is finite,
\item[$(ii)$] $P_{\F}:=P(f)\cap\F(f)$ is compact.
\end{itemize}
\end{Definition} 

It follows again from classical results that if $f$ is subhyperbolic, 
then $\F(f)$ is either empty or consists of attracting basins, and by 
condition $(ii)$ there can be only finitely many attracting cycles. 

By condition $(i)$, every singular value of $f$ in $\J(f)$ 
is preperiodic and every critical point in 
$\J(f)$ is strictly preperiodic.
It also follows that $f$ has no Cremer points, 
i.e., irrationally indifferent periodic points that lie in the Julia set.
This means that 
each singular value in $\J(f)$ eventually lands on some repelling orbit 
(for a proof of the above statements see e.g. \cite[Proposition 2.5]{mihaljevic_brandt}). 

A rational map is said to be subhyperbolic if it is
\emph{expanding} with respect to an orbifold metric. 
This characterization 
is equivalent to saying that every critical orbit is finite or converges
to an attracting periodic orbit.
Note that the postsingular set of a subhyperbolic rational function 
does not have to be bounded.  
However, for an \emph{arbitrary} subhyperbolic 
transcendental map it is not possible to 
define an orbifold metric on a neighbourhood of the Julia which is expanded by the map. 
As we will see later, strongly subhyperbolic maps are exactly those maps
for which we can construct the required orbifold metric.
For completeness of this section, let us
recall the definition of a strongly subhyperbolic map.
\begin{Definition}
\label{dfn_tame}
A subhyperbolic transcendental entire map $f$ 
is called \emph{strongly subhyperbolic}, if  
$\J(f)\cap A(f)=\emptyset$ 
and there is a constant $R<\infty$ such that 
$\deg(f,z)<R$ holds for all $z\in\J(f)$.
\end{Definition}

\section{Subhyperbolic maps and dynamically associated orbifolds}
\label{sec associate}
Let $f$ be a strongly subhyperbolic map. 
The first step towards the proof of Theorem \ref{thm_maintheorem} 
is to find hyperbolic orbifolds
\begin{eqnarray*}
\Or_f=(S_f,\nu_f) \quad\text{and}\quad \Ort_f=(\widetilde{S}_f,\tilde{\nu}_f)
\end{eqnarray*}
such that $f:\Ort_f\to\Or_f$ is expanding with respect to 
the hyperbolic metric on $\Or_f$. We will make use of 
the following simple observation.
\begin{Proposition}
\label{expansion}
Let $\Or=(S,\nu)$ and $\Ort=(\widetilde{S},\tilde{\nu})$ 
be hyperbolic orbifolds with metrics $\rho_{\Or}(z)\vert dz\vert$ and 
$\rho_{\Ort}(z)\vert dz\vert$, respectively. 
Let $f:\Ort\rightarrow\Or$ be a covering map and assume that 
the inclusion $\Ort\hookrightarrow\Or$ is holomorphic but not a covering. Then 
\begin{eqnarray*}
\Vert Df(z)\Vert_{\Or} 
:= \vert f^{'}(z)\vert \cdot\frac{\rho_{\Or}(f(z))}{\rho_{\Or}(z)}>1,
\end{eqnarray*}
wherever this is defined.
\end{Proposition}

\begin{proof}
By Theorem \ref{pick} $f$ is a local isometry, hence 
\begin{eqnarray*}
\rho_\Ort(z) = \rho_\Or(f(z))\cdot\vert f^{'}(z)\vert.
\end{eqnarray*}
Since the inclusion is only holomorphic, it is a strict contraction, and hence 
$\rho_\Ort(z)>\rho_\Or(z)$. So altogether,
\begin{eqnarray*}
\rho_\Ort(z) = \rho_\Or(f(z))\cdot\vert f^{'}(z)\vert >\rho_\Or(z),
\end{eqnarray*}
implying that $\Vert Df(z)\Vert_\Or>1$.
\end{proof}

\subsection{Construction of \texorpdfstring{$\Or_f$}{O_f} and \texorpdfstring{$\Ort_f$}{O_f} for a strongly subhyperbolic map \texorpdfstring{$f$}{f}}
In order to find 
hyperbolic orbifolds $\Or_f$ and $\Ort_f$ 
such that the assumptions of 
Proposition \ref{expansion} are satisfied, 
we will roughly follow the approach of 
Douady and Hubbard for subhyperbolic rational 
maps \cite{douady_hubbard}.
We need to be able to compute sufficiently good 
estimates of the orbifold metrics on $\Or_f$ and $\Ort_f$.
Our requirements are formalized in the following proposition.

\begin{dfn&prop}[Dynamically associated orbifolds]
\label{prop_Or}
Let $f$ be a strongly subhyperbolic function. Then there exist
orbifolds $\Or_f=(S_f,\nu_f)$ and $\Ort_f=(\widetilde{S}_f,\tilde{\nu}_f)$
with the following properties:
\begin{itemize}
\item[$(a)$] The set $B_f$ of ramified points of $\Or_f$ is a finite set 
that contains $P_{\J}$. Furthermore, there exists a point $p\in\Or_f\setminus S(f)$
such that $\nu_f(p)=2\cdot k$ for some $k\geq 1$.
\item[$(b)$] $\J(f)\subset\Or_f$ while $P_{\F}\cap\Or_f=\emptyset$.
\item[$(c)$] $\Or_f$ is a hyperbolic orbifold containing a punctured neighbourhood 
of $\infty$.
\item[$(d)$] $f:\Ort_f\to\Or_f$ is a covering map.
\item[$(e)$] The inclusion $\Ort_f\hookrightarrow \Or_f$ is holomorphic but not a covering map.
Furthermore, if $S_f\neq\C$, then $\overline{\widetilde{S}}_f\subset S_f$.
\end{itemize}
We say that the pair $(\Ort_f,\Or_f)$
of Riemann orbifolds is \emph{dynamically associated to $f$}, if $\Ort_f$ and
$\Or_f$ satisfy $(a)$-$(e)$.
\end{dfn&prop}

\begin{proof}
We start with the construction of $S_f$.  
If $\F(f)=\emptyset$, then define $S_f:=\C$. 
Otherwise, the Fatou set of $f$ consists of attracting basins only 
and, as in the proof of Proposition \ref{prop_dt}, 
we can find a bounded neighbourhood $U$ of the set 
$P_{\F}$ such that $\overline{U}\subset\F(f)$, 
$\C\setminus U$ is connected and $\overline{f(U)}\subset U$. 
We choose a set $U$ with this property and define
\begin{align}
\label{eqn_S_f}
S_f:= \C\backslash\overline{U}.
\end{align}
Note that $S_f$ is connected and that $\J(f)$ is entirely contained in $S_f$.

Now assume that there is a point $p\in P_{\J}\setminus S(f)$, 
such that for every point $z\in\Crit(f)$ with $f^n(z)=p$ there exists 
$k\geq 1$ with $\deg(f^n,z)=2\cdot k$. Then we define the ramification value
of a point $z\in S_f$ to be 
\begin{eqnarray}
\label{ram}
\nu_f(z):=\lcm\lbrace \deg(f^m,w),\text{ where } f^m(w)=z\rbrace.
\end{eqnarray}
If there is no point $p$ with such a property, then pick a repelling 
fixed point $p\notin P(f)$ of $f$. Observe that such a point exists,  
since every map with a bounded set of singular values has infinitely many 
fixed points \cite{langley_zheng}, and since $f$ is subhyperbolic, 
only finitely many of them can be non-repelling. 
Since $p\in\J(f)$, it also belongs to $S_f$ and we 
define the ramification value of every point $z\in S_f\setminus\{ p\}$ to be 
the value defined by equation (\ref{ram}), and 
assign $p$ the ramification value $\nu_f(p)=2$.
Observe that in both cases, there is a point $p\in S_f$ such that 
$\nu_f(p)$ is a multiple of $2$.

Let $\Or_f=(S_f,\nu_f)$. Since $\nu_f(z)>1$ if and only if $z$ belongs to 
$P_{\J}\cup\lbrace p\rbrace$, the set of ramified points 
of $\Or_f$ is finite.  
Furthermore, no critical point $c\in S_f$ belongs to a periodic cycle 
and since we have assumed that the local degree at all 
points in $\J(f)$ is globally bounded by some constant $R$, 
the ramification value $\nu_f(z)$ is necessarily a finite number 
for each $z\in S_f$. Hence $\Or_f=(S_f,\nu_f)$ is a Riemann orbifold 
and statement $(a)$ follows by construction.

Note that part $(b)$ is an immediate consequence of the definition of $S_f$.

Next we will prove that $\Or_f$ is hyperbolic. 
Observe that it is sufficient to restrict to
the case when $p\in P_{\J}$, since every orbifold that 
is holomorphically included in a hyperbolic orbifold has to be hyperbolic 
as well. So we consider the orbifold $\Or_f$ with $B_f=P_{\J}$. 
We will give a proof by contradiction, so let us assume that 
$\Or_f$ is not hyperbolic. 
Since $S_f\subset\C$, it follows that $\Or_f$ must be parabolic. 
By assumption $P_{\J}\neq\emptyset$, so it follows from
\cite[Theorem A4]{mcmullen_2} that $\Or_f$ must be isomorphic to 
$\C$ with signature either $(n)$ or $(2,2)$. 
This implies that $\F(f)=\emptyset$, 
hence all singular values of $f$ belong to $\J(f)$. 

Assume first that $\Or_f$ is isomorphic to 
$\C$ with signature $(n)$. Then $f$ has only one singular value, 
say at $z=0$. It follows from a covering space argument that 
$f(z)=\exp(az+b)$, where $a\in\C^{*}$ and $b\in\C$. 
But then $0$ is an asymptotic  value (and since it is omitted, 
it cannot be a fixed point either), yielding a contradiction. 

So $\Or_f$ must be isomorphic to $\C$ with signature $(2,2)$.
By the previous argument, $f$ has exactly two singular values 
$v_1$ and $v_2$ which are necessarily critical values, and any of 
their preimages is either a critical point of local degree 
two or a regular point. Signature $(2,2)$ also implies that $P(f)=S(f)$, 
meaning that both critical values are either fixed points of $f$ 
or they form a two-cycle. Since $f$ is subhyperbolic but 
not hyperbolic, $v_1$ and $v_2$ are both repelling. 
In particular, $\deg(f,v_1)=\deg(f,v_2)=1$. 
Let $\Or_f^{'}$ be the orbifold which has exactly 
the regular preimages of $v_1$ and $v_2$ as ramified points, 
assigning them the ramification value two. Clearly, 
$\nu_f^{'}(v_1)=\nu_f^{'}(v_2)=2$. Then $f: \Or_f^{'}\rightarrow \Or_f$ 
is a covering map and since $\Or_f$ is parabolic, 
so is $\Or_f^{'}$, which means that $v_1$ and $v_2$ are the only 
ramified points in $\Or_f^{'}$. Hence $\Or_f^{'}=\Or_f$. 
By conformal conjugacy we can assume that $v_1=1$ and $v_2=-1$. 
Then the map $\C\rightarrow\Or_f,\; z\mapsto \cos(z)$ 
is a universal covering map. Since $f:\Or_f\rightarrow\Or_f$ 
is a covering map, it lifts to a conformal $\C$-isomorphism 
$g(z)=az +b$, $a\neq 0$, yielding the relation
\begin{eqnarray*}
f(\cos(z)) = \cos(az+b).
\end{eqnarray*}
By periodicity and symmetry of the cosine map, $a\in\Z$ and $b\in\pi\Z$. 
But this means that $f$ or $-f$ is a Chebyshev polynomial, contradicting the fact that 
$f$ is transcendental. Hence $\Or_f$ is hyperbolic.

By construction, $\Or_f$ is the complement of a, possibly empty, 
compact set, hence $(c)$ follows.

Define
\begin{eqnarray*}
\widetilde{S}_f:=f^{-1}(S_f)
\end{eqnarray*}
and
\begin{eqnarray*}
\tilde{\nu}_f(z):\widetilde{S}_f\to\N,\quad z\mapsto\frac{\nu_f(f(z))}{\deg(f,z)}.
\end{eqnarray*}
By equation (\ref{ram}), $\tilde{\nu}_f(z)$ is a positive integer for every 
$z\in \widetilde{S}_f$ and by the Identity Theorem, the set of 
points $z\in\widetilde{S}_f$ with $\tilde{\nu}_f(z)>1$ is discrete. 
Hence $\Ort_f=(\widetilde{S}_f,\tilde{\nu}_f)$ is a Riemann orbifold.

Since $A(f)\cap S_f=\emptyset$, the map $f:\widetilde{S}_f\rightarrow S_f$ 
is a branched covering. Furthermore, 
\begin{eqnarray*}
\deg(f,z)\cdot\tilde{\nu}_f(z)=\deg(f,z)\cdot\frac{\nu_f(f(z))}{\deg(f,z)}=\nu_f(f(z)),
\end{eqnarray*}
hence $f:\Ort_f\rightarrow\Or_f$ is an orbifold covering map, proving statement $(d)$. 

We will show now that the inclusion $\Ort_f\hookrightarrow\Or_f$ is 
holomorphic but not a covering. 
First note that $\widetilde{S}_f\subset S_f$ by construction of $S_f$.
Moreover, if $S_f\neq\C$, then $\widetilde{S}_f$ is a 
relatively compact subset of $S_f$ (see equation (\ref{eqn_S_f})).
Recall that by $(a)$, there is a point $p\in\Or_f\setminus S(f)$ such that
$\nu_f(p)$ is a multiple of $2$. 
The fact that $p\not\in S(f)$ implies that $p$ has infinitely many
preimages $p_i$ under $f$, and for every such preimage point 
we have $\deg(f,p_i)=1$. 
Moreover, $\nu_f(p_i)=1$ holds for all but 
finitely many of the preimages of $p$, since by $(a)$, 
$\Or_f$ has only finitely many ramified points.
On the other hand, $\tilde{\nu}_f(p_i)=2$, which means that  
$\tilde{\nu}_f(p_i) = 2>\nu_f(p_i)=1$,
hence the inclusion is not a covering map. 

Let $z\in S_f$. Observe that the definition of $\nu_f$ (see equation (\ref{ram})) 
together with the fact that 
for any point $\omega\in\C$ the local degree of an iterate $f^m$ of $f$ is given by  $\deg(f^m,\omega)=\deg(f,\omega)\cdot\deg(f,f(\omega))\cdot \ldots\cdot\deg(f,f^{m-1}(\omega))$ 
implies that $\nu_f(z)\cdot\deg(f,z)$ divides $\nu_f(f(z))$. 
Since $f:\Ort_f\to\Or_f$ is a covering, 
$\tilde{\nu}_f(z)=\nu_f(z)\cdot\deg(f,z)=\nu_f(f(z))$. 
Hence $\nu_f(z)$ divides $\tilde{\nu}_f(z)$ and this proves that the inclusion 
$\Ort_f\hookrightarrow\Or_f$ is a holomorphic map.
\end{proof}

\begin{Remark}
Let $(\Ort_f,\Or_f)$ be dynamically associated to $f$. 
Note that $\Ort_f$ is 
usually not connected. 
However, If $f$ has finitely many tracts over $\infty$, e.g. if $f$ has 
finite order, then the number of components of $\Ort_f$ is finite.
Observe also that it follows from Proposition \ref{prop_Or}$(e)$ 
that the set $B_f$ of ramified points of $\Or_f$ satisfies $f(B_f)\subset B_f$. 

Finally we would like to mention 
that we have proved more than the existence of a hyperbolic 
orbifold $\Or_f$ satisfying the remaining assumptions of 
Proposition \ref{prop_Or}. In fact, we have also shown the following. 
\emph{Let $f$ be a transcendental entire function and let $\Ort$ and $\Or$ be 
any two orbifolds such that $f:\Ort\to\Or$ is a covering map. Then $\Or$ (and 
hence $\Ort$) is hyperbolic.}
\end{Remark}

\begin{Corollary}
\label{cor_(f)}
Let $f$ and $\Ort_f$ be as in Proposition \ref{prop_Or}.  
Then there is a constant $K>1$ and an infinite sequence of points $z_i$ 
for which $\tilde{\nu}_f(z_i)$ is a multiple of $2$,
such that 
$\vert z_i\vert<\vert z_{i+1}\vert\leq K\vert z_i\vert$ holds for all $i$.
\end{Corollary}

\begin{proof} 
Let $p\in\Or_f$ be a point such that 
$\nu_f(p)$ is a multiple of $2$ and 
let $\gamma$ be a Jordan curve in $\C$ such that the bounded component 
of $\C\setminus\gamma$ contains $S(f)$ but not $p$. 
The components of the preimage of the unbounded component of $\C\setminus\gamma$
are tracts of $f$ and every such tract contains infinitely many preimages
of $p$. Let $z_i$ denote the preimages of $p$ lying in one such 
(arbitrary but fixed) tract.
Using standard estimates on the hyperbolic metric in a 
simply-connected domain \cite[Corollary A.8]{milnor}, one 
easily deduces that there is a constant $K>1$ such that 
$\vert z_i\vert <\vert z_{i+1}\vert\leq K\vert z_i\vert$ holds 
for infinitely many $i$.
(For details, see e.g. \cite[Proof of Lemma 5.1]{rempe_5} 
or \cite[Proof of Proposition 3.4]{mihaljevic_brandt}.) 
However, since all but finitely 
many $z_i$ are regular points of $f$, it follows that $\tilde{\nu}_f(z_i)=\nu_f(p)$
and this is, by $(a)$, a multiple of $2$.
\end{proof}

\begin{notations}
 For a pair $(\Ort_f,\Or_f)$ of orbifolds dynamically associated 
to $f$, we denote by $\tilde{\rho}_f$ and $\rho_f$ 
the densities of the hyperbolic metrics of $\Ort_f$ and $\Or_f$, respectively.
\end{notations}

We conclude this section with 
the following simple observation which justifies our
restriction to strongly subhyperbolic maps.

\begin{Proposition}
Let $f$ be a subhyperbolic map for which there is a pair of 
dynamically associated orbifolds. 
Then $f$ is strongly subhyperbolic.
\end{Proposition}

\begin{proof}
If $a$ is an asymptotic value of $f$, 
then for any compact set $K\subset\C$ containing $a$, there 
exists a component of $f^{-1}(K)$ which is not compact. Hence 
there is no domain $U\ni a$ such that $f: f^{-1}(U)\to U$ is a 
proper map. Hence, if an asymptotic value of $f$ 
belongs to $\J(f)$, then there is no domain $U\supset\J(f)$ such 
that $f:f^{-1}(U)\to U$ is a (branched) covering map.  

Assume now that $f$ has a critical value $w\in\J(f)$, such that for every $n\in\N$ 
there exists a point $z_n$ with $f(z_n)=w$ and $\deg(f,z_n)\geq n$. 
If there was a pair $(\Ort_f,\Or_f)$ of dynamically associated orbifolds,  
then Proposition \ref{prop_Or}$(d)$ would imply that $w$ is a puncture of $\Or_f$,
contradicting the fact that $\J(f)\subset S_f$. 

\end{proof}

\section{Uniform expansion}
\label{sec uniform}
Let $f$ be a strongly subhyperbolic map and let 
$(\Ort_f,\Or_f)$ be dynamically associated to $f$.
By Proposition \ref{expansion}, 
\begin{eqnarray*}
\Vert Df(z)\Vert_{\Or_f} = \vert f^{'}(z)\vert \cdot\frac{\rho_f(f(z))}{\rho_f(z)}>1
\end{eqnarray*}
wherever defined, so in particular for all $z\in\Ort_f$. 
\begin{Remark}
If $w\in\Ort_f$ is a point such that $\nu_f(f(w))>1$, then 
it follows by Proposition \ref{prop_Or}$(d)$,$(e)$
that $\nu_f(w)\cdot\deg(f,w)$ divides $\nu_f(f(w))$. In this case we define
\begin{eqnarray*}
\Vert Df(w)\Vert_{\Or_f}:=\lim_{z\rightarrow w} \Vert Df(z)\Vert_{\Or_f},
\end{eqnarray*}
and so if $\nu_f(w)\cdot\deg(f,w)=\nu_f(f(w))$, 
then $\Vert Df(w)\Vert_{\Or_f}$ is a finite number, 
while otherwise $\Vert Df(w)\Vert_{\Or_f}=\infty$. 
\end{Remark}

Our goal is to show that the expansion of $f$ is uniform.
\begin{Theorem}[Uniform expansion]
\label{thm_uniform_expansion}
Let $f$ be strongly subhyperbolic and let $(\Ort_f,\Or_f)$ be 
dynamically associated to $f$. Then 
there is a constant $E>1$ such that 
\begin{eqnarray*}
\Vert Df(z)\Vert_{\Or_f} \geq E
\end{eqnarray*}
for all $z\in\Ort_f$.
\end{Theorem}

The remainder of Section \ref{sec uniform} is devoted to the proof of Theorem \ref{thm_uniform_expansion}.

\subsection{Continuity of orbifold metrics}
We want to show the following continuity statement: 
Let $\Or$ be a Riemann orbifold, let $p$ be a regular and 
$q$ a ramified point of $\Or$. Then the value of the density map $\rho_{\Or}$
of the orbifold metric at $p$ depends
continuously on $q$, i.e., if we perturb the point $q$ slightly, then the density of 
the corresponding orbifold metric at the point $p$ will also undergo only 
a small change. 

This statement is surely not new but we were not able to locate a reference.
Hence we include a proof for completeness. We will restrict to orbifolds 
whose underlying surface is a subset of the sphere; the proof in 
the general case, where the underlying surface of $\Or$ is an 
arbitrary Riemann surface, can be derived using exactly the
same arguments, with the additional step of taking charts. 

So let $S\subset\widehat{\C}$ and let $\Or=(S,\nu)$.  
Denote by $B$ the set of ramified points of $\Or$. 
Let $n>1$ be an arbitrary but fixed integer. 
For a point $q\in S\backslash B$ we define a new orbifold  $\Or_q=(S,\nu_q)$, where
\begin{align*}
\nu_q(z)=\begin{cases} \nu(z)&\text{ if } z\neq q,\\ n &\text{ if } z=q.
         \end{cases}
\end{align*}
Furthermore, we assume that every such orbifold has a universal cover
(hence we exclude the case when $S$ is a sphere and the set $B$ is either 
empty or consists of only one point with ramification value $m\neq n$). 
Note that any two such orbifolds $\Or_q$ and $\Or_{\tilde{q}}$ 
have the same signature and hence the same uniformized universal cover. 
Let us denote the density of the orbifold metric on $\Or_q$ by $\rho_q$.

For a point $p\in S\setminus B$ we define the map
\begin{align*}
M_p: S\setminus (B\cup\{ p\})\rightarrow (0,\infty],\quad q\mapsto\rho_q(p).
\end{align*}

\begin{Theorem}[Continuity of orbifold metrics]
\label{thm_continuity}
Let $p\in S\setminus B$ be arbitrary but fixed. 
Then the map $M_p$ is continuous at every point in $S\setminus(B\cup\{ p\})$.
\end{Theorem}

\begin{proof}
Let $q^{*}\in S\setminus(B\cup\{ p\})$ be an arbitrary but fixed point. 
We want to show that $M_p$ is continuous at $q^{*}$. 

Pick a sufficiently small Jordan domain $D\ni q^{*}$ such 
that $\overline{D}\cap \left(B\cup\{ p\}\right)=\emptyset$.
Let $a$ be a point in $S\setminus (\overline{D}\cup B\cup\{p\})$. 
By conformal conjugacy, we can assume that $a=0$.
 
For two points $z_1,z_2\in D$ let us denote by $d_D(z_1,z_2)$ the distance between 
$z_1$ and $z_2$ measured in the hyperbolic metric of $D$. 
For a point $q\in D$ consider the unique Riemann map $H_q: D\to\H$ which maps 
$q^{*}\mapsto i$ and $q\mapsto h_q i$, where $h_q:=\e^{d_D(q^{*},q)}$. 
Let $L_q:\H\to\H$, $(x+ iy)\mapsto x + h_q y i$. Then $L_q$ is a $h_q$-quasiconformal 
self-map of $\H$. Define 
\begin{align*}
\varphi_q:D\to D,\quad z\mapsto H_q\circ L_q\circ H_q^{-1}(z).
\end{align*}
 
It is easy to see that $\varphi_q$ extends continuously to the complement 
of $D$ as the identity map and that the extended map, which we will also 
denote by $\varphi_q$, is a $h_q$-quasiconformal map 
(see e.g. \cite[Lemma 5.2.3]{graczyk_swiatek}).
Observe that $\varphi_q\to\varphi_{q^{*}}\equiv \id$ as $q\to q^{*}$.

Let $C$ be the uniformized universal 
covering surface of $\Or_{q^{*}}$ and $\Or_q$ 
and let $\pi_{q^{*}}:C\to \Or_{q^{*}}$ and $\pi_q:C\to \Or_q$ 
be universal covering maps, 
both normalized such that $\pi_{q^{*}}(0)=\pi_q(0)=0$ 
and $\pi_{q^{*}}^{'}(0)=\pi_q^{'}(0)$. 
Considered as a map between orbifolds, 
$\varphi_q:\Or_{q^{*}}\to\Or_q$ is a homeomorphism and hence can be lifted to a 
homeomorphism on $C$.  
From now on we will assume that $C=\D$ since the other two cases follow by the same 
strategy, using even simpler calculations. 

\begin{Claim}
 There is a unique lift $\tilde{\varphi}_q:\D\to\D$ of $\varphi_q$ 
such that $\tilde{\varphi}_q(0)=0$.
\end{Claim}

\begin{proof of claim}
Let $G_q$ denote the covering group of $\D$ over $\Or_q$ and assume that 
there exist two distinct lifts
$\tilde{\varphi}_q$, $\tilde{\tilde{\varphi}}_q$ 
of $\varphi_q$ that fix $0$. 
There exists a mapping $h\in G_q$ such that
$\tilde{\tilde{\varphi}}_q(z) = h(\tilde{\varphi}_q(z))$ holds for 
all $z\in\D$. It follows from our assumption that $h(0)=0$, hence 
$h\in\Stab(0)\subset G_q$, where 
$\Stab(0)$ denotes the stabilizer of $0$ in $G_q$.
But $\pi_q(0)=0$ and $0$ is a non-ramified point 
of $\Or_q$, which means that $\Stab(0)\subset G_q$ is trivial. 
Hence $h\equiv\id$ and $\tilde{\varphi}_q\equiv\tilde{\tilde{\varphi}}_q$. 
\end{proof of claim}

We have the following commutative diagram:
\begin{diagram}
\D         &\rTo^{\tilde{\varphi}_q} &\D\\
\dTo^{\pi_{q^{*}}} &                   &\dTo_{\pi_q}\\
\Or_{q^{*}}   &\rTo_{\varphi_q}           &\Or_q
\end{diagram}
Since $\varphi_q$ is a $h_q$-quasiconformal map 
and $\pi_{q^{*}}$ and $\pi_q$ are holomorphic,
the map $\tilde{\varphi}_q:\D\to\D$ is also $h_q$-quasiconformal. Moreover, 
$\tilde{\varphi}_q$ is conformal when restricted to the set 
$\Omega:=\D\setminus\pi_{q^{*}}^{-1}(D)$, so in particular in  
a sufficiently small neighbourhood of any point in the set $\{\pi_{q^{*}}^{-1}(p)\}$.

Furthermore, the lifts $\tilde{\varphi}_q$ converge to $\tilde{\varphi}_{q^{*}}$ 
(locally uniformly) as $q\to q^{*}$ and, due to the chosen normalization,  
$\tilde{\varphi}_{q^{*}}\equiv\id\vert_{\D}$. 
Moreover, when restricted to $\Omega$, the maps converge in the $C^1$-norm, hence 
$(\tilde{\varphi}_q)^{'}\vert_{\Omega}\to(\tilde{\varphi}_{q^{*}})^{'}\vert_{\Omega}$ 
when $q\to q^{*}$.

By the above diagram we can write
$\pi_q(z) = (\varphi_q\circ\pi_{q^{*}}\circ\tilde{\varphi}_q^{-1})(z)$
for every $z\in\D$. 
Recall that if $w\in S$ and $z_q\in\{\pi_q^{-1}(w)\}$, 
then the value of the density function $\rho_q$ at $w$ is given by 
$\rho_q(w)=\rho_{\D}(z_q)\cdot (\pi_q^{'}(z_q))^{-1}$ and  
this does not depend on the choice of the preimage of $w$. 
Similarly, if $z_{q^{*}}\in\{\pi_{q^{*}}^{-1}(w)\}$, 
then $\rho_{q^{*}}(w)=\rho_{\D}(z_{q^{*}})\cdot(\pi_{q^{*}}^{'}(z_{q^{*}}))^{-1}$. 
Hence, 
\begin{align*}
\vert\rho_{q^{*}}(w) - \rho_q(w)\vert = 
\left| \frac{\rho_{\D}(z_{q^{*}})}{\pi_{q^{*}}^{'}(z_{q^{*}})} - \frac{\rho_{\D}(z_q)}{\pi_q^{'}(z_q)}\right|.
\end{align*} 

Observe first that $\varphi_q(p)=p$, since $p\in S\setminus D$. 
Let us fix a point $p_q\in\{\pi_q^{-1}(p)\}$. Then
\begin{align}
\label{eqn_1}
p=\pi_q(p_q) = (\varphi_q\circ\pi\circ\tilde{\varphi}_q^{-1})(p_q) = (\pi\circ\tilde{\varphi}_q^{-1})(p_q).
\end{align}
Let $p_{q^{*}}\in\{\pi_{q^{*}}^{-1}(p)\}$ be the unique point such that 
$p_q=\tilde{\varphi}_q(p_{q^{*}})$.
We obtain
\begin{eqnarray*}
\pi_q^{'}(p_q) &=& 
\varphi_q^{'}((\pi_{q^{*}}\circ\tilde{\varphi}_q^{-1})(p_q))
\cdot\pi_{q^{*}}^{'}(\tilde{\varphi}_q^{-1}(p_q))
\cdot(\tilde{\varphi}_q^{-1})^{'}(p_q)\\ 
&\underbrace{=}_{\left(\ref{eqn_1}\right)}& 
\varphi_q^{'}(p)\cdot\pi_{q^{*}}^{'}(p_{q^{*}})\cdot(\tilde{\varphi}_q^{-1})^{'}(p_q)
\underbrace{=}_{\varphi_q^{'}(p)=1} \pi_{q^{*}}^{'}(p_{q^{*}})\cdot(\tilde{\varphi}_q^{-1})^{'}(p_q). 
\end{eqnarray*}
Hence
\begin{eqnarray*}
\vert \rho_{q^{*}}(p) - \rho_q(p)\vert 
&=& \frac{1}{\vert \pi_{q^{*}}^{'}(p_{q^{*}})\vert}\cdot \left|\rho_{\D}(p_{q^{*}}) - \frac{\rho_{\D}(\tilde{\varphi}_q(p_{q^{*}}))}{(\tilde{\varphi}_q^{-1})^{'}(p_q)}\right|\\
&=& \frac{1}{\vert\pi_{q^{*}}^{'}(p_{q^{*}})\vert}\cdot
\left|\frac{1}{1-\vert p_{q^{*}}\vert^2} - 
\frac{1}{(\tilde{\varphi}_q^{-1})^{'}(p_q)\cdot(1-\vert\tilde{\varphi}_q(p_{q^{*}})\vert^2)}\right|.
\end{eqnarray*}
Since $\tilde{\varphi}_q\to\id$ in the $C^1$-norm in a neighbourhood of 
$p_{q^{*}}$ when $q\to q^{*}$, the expression 
$(\tilde{\varphi}_q^{-1})^{'}(p_q)\cdot(1-\vert\tilde{\varphi}_q(p_{q^{*}})\vert^2)$ tends to
$1-\vert p_{q^{*}}\vert^2$ and hence $\vert \rho_{q^{*}}(p) - \rho_q(p)\vert \to 0$.
\end{proof}

\subsection{Estimates of metrics with infinitely many singularities}
We are now able to prove Theorem \ref{thm_main2}, which is 
the key-statement for the proof of Theorem \ref{thm_uniform_expansion}.
\begin{Theorem}
\label{thm_orb_main}
Let $K>1$ and let $z_i$, $i\in\N$, be an infinite sequence of points satisfying 
 $\vert z_i\vert <\vert z_{i+1}\vert\leq K\vert z_i\vert$.
Let $\Or=(\C,\nu_{\Or})$, where
\begin{eqnarray*}
\nu_{\Or}(z)=
 \begin{cases}
2&\text{ if } z=z_i\text{ for some i},\\
1&\text{ otherwise.}
 \end{cases}
\end{eqnarray*}
Then the density $\rho_{\Or}$ of the hyperbolic metric
on $\Or$ satisfies
\begin{align*}
\rho_{\Or}(z)\geq\Oo\left(\frac{1}{\vert z\vert}\right)\quad\text{as}\quad z\rightarrow\infty.
\end{align*}
\end{Theorem}
\begin{proof}
First observe that, by affine conjugacy, we can assume that 
$0$ is one of the ramified points of $\Or$. 

Let $z\neq z_i$ be an arbitrary but fixed point in $\Or$. Depending on $z$, we choose 
$b=b(z)=z_k$, where $z_k$ satisfies $\vert z_k\vert\geq 2\vert z\vert$ and is minimal
with this property, i.e., if $\vert z_j\vert<\vert z_k\vert$ then 
$\vert z_j\vert <2\vert z\vert$.
It follows immediately that
\begin{align}
\label{eqn_b}
2\vert z\vert\leq\vert b\vert=\vert z_k\vert\leq K\vert z_{k-1}\vert\leq 2K\vert z\vert.
\end{align}

As next we set $c=c(z)=z_l$, where $z_l$ is minimal with the 
property $\vert z_l\vert\geq 2\vert b\vert$. We then obtain
\begin{align}
\label{eqn_c}
4\vert z\vert\leq 2\vert b\vert\leq\vert c\vert = \vert z_l\vert\leq K\vert z_{l-1}\vert
\leq 2K\vert b\vert\leq 4K^2\vert z\vert.
\end{align}
For any three pairwise distinct points $p,q,r\in\C$ denote by 
$\Or_{p,q,r}:=(\C,\nu_{p,q,r})$ the orbifold defined by 
\begin{eqnarray*}
\nu_{p,q,r}(w)=\begin{cases} 2 &\text{ if } w\in\{ p,q,r\},\\ 1 &\text{ otherwise}.\end{cases}
\end{eqnarray*}
Note that every such orbifold is hyperbolic, since its Euler characteristic 
equals $-1/2$. 
We denote by $\rho_{p,q,r}$ the density of the hyperbolic metric on $\Or_{p,q,r}$.

Observe first that $\Or$ is holomorphically embedded in $\Or_{0,b,c}$, 
 and it follows from 
Theorem \ref{pick} that $\rho_{\Or}(w)>\rho_{0,b,c}(w)$ holds for all $w\in\Or$.
Let $\tilde{b}=\tilde{b}(z):=b/\vert z\vert$ 
and $\tilde{c}=\tilde{c}(z):=c/\vert z\vert$. 
Then the map   
\begin{align*}
 S_z:\Or_{0,b,c}\to \Or_{0,\tilde{b},\tilde{c}},\quad w\mapsto\frac{w}{\vert z\vert}
\end{align*}
is obviously a conformal isomorphism and hence a local isometry.
Altogether, we obtain 
\begin{eqnarray}
\label{eqn_abc}
\rho_{0,\tilde{b},\tilde{c}}(S(w)) &=& 
\rho_{0,b,c}(w)\cdot\vert w\vert<\rho_{\Or}(w)\cdot\vert w\vert.
\end{eqnarray}
Let $\tilde{z}:=S(z)=z/\vert z\vert$. Then equations (\ref{eqn_b}) and (\ref{eqn_c}) 
yield 
\begin{align*}
2\leq\vert\tilde{b}\vert\leq 2K\quad\text{and}\quad 4\leq\vert \tilde{c}\vert\leq 4K^2,
\end{align*}
i.e., $\tilde{b}\in A_1=A_1(K):=\{w:2\leq\vert w\vert\leq 2K\}$ and 
$\tilde{c}\in A_2=A_2(K):=\{w: 4\leq\vert w\vert\leq 4K^2\}$ 
belong to compact annuli disjoint from $\tilde{z}$ (see Figure \ref{figure_annuli}).  

By Theorem \ref{thm_continuity} the map
\begin{align*}
D^2: \C\setminus\{\tilde{z}\} \times \C\setminus\{\tilde{z}\} \to (0,\infty),\quad 
(x,y)\mapsto \rho_{0,x,y}(\tilde{z})
\end{align*}
is a composition of two continuous maps and hence itself continuous. Furthermore,
it attains its minimum (and maximum) on the compact set $A_1\times A_2$. Hence, 
there exist constants $0<m(K),M(K)<\infty$ depending only on $K$ such that
\begin{align*}
m(K)<\rho_{0,\tilde{b},\tilde{c}}(\tilde{z})<M(K).
\end{align*}
By setting $w=z$ in equation (\ref{eqn_abc}), we finally get
\begin{align*}
m(K)\cdot\frac{1}{\vert z\vert}\leq \frac{\rho_{0,\tilde{b},\tilde{c}}(\tilde{z})}{\vert z\vert}<\rho_{\Or}(z),
\end{align*}
and the assertion of the theorem follows.
\end{proof}

\begin{figure}
\centering
\psfrag{a}{$0$}
\psfrag{b}{$\tilde{b}$}
\psfrag{c}{$\tilde{c}$}
\psfrag{1}{$1$}
\psfrag{2}{$2$}
\psfrag{4}{$4$}
\psfrag{w}{$\tilde{z}$}
\psfrag{A1}{$A_1$}
\psfrag{A2}{$A_2$}
\includegraphics[width=.5\linewidth]{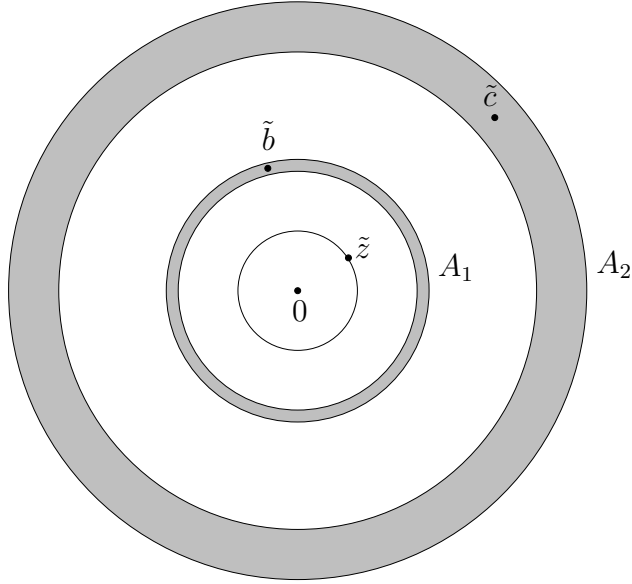}
\caption{The parameters $\tilde{b}$ and $\tilde{c}$ belong to the 
compact annuli $A_1$ and $A_2$.}
\label{figure_annuli}
\end{figure}

Theorem \ref{thm_orb_main} and Corollary \ref{cor_(f)} 
immediately imply the following.

\begin{Corollary}
\label{estimate_rhotilde}
Let $f$ be strongly subhyperbolic and let 
$(\Ort_f,\Or_f)$ be dynamically associated to $f$. 
Then
\begin{eqnarray*}
\tilde{\rho}_f(z)\geq \Oo\left( \frac{1}{\vert z\vert}\right) \quad\text{as}\quad z\rightarrow\infty,
\end{eqnarray*}
where $\tilde{\rho}_f(z)$ denotes the density of the hyperbolic metric on $\Ort_f$.
\end{Corollary}

\subsection{Proof of uniformity}
Using Theorems \ref{thm_orb_main} and \ref{thm_continuity},
we can finally deduce that $f:\Ort_f\to\Or_f$ is a uniform expansion
with respect to the hyperbolic metric of $\Or_f$.

\begin{proof}[Proof of Theorem \ref{thm_uniform_expansion}] 
Let $\tilde{\rho}_f$ and $\rho_f$ denote the densities of the 
hyperbolic metrics on $\Ort_f$ and $\Or_f$, respectively. 
Since $f:\Ort_f\rightarrow\Or_f$ is a covering map, 
our claim is equivalent to the statement 
that there is a constant $E>1$ such that 
\begin{eqnarray*}
\frac{\tilde{\rho}_f(z)}{\rho_f(z)}\geq E>1.
\end{eqnarray*}
If $\F(f)\neq\emptyset$, 
recall that by Proposition \ref{prop_Or}$(e)$, 
$\Or_f$ is modelled on a 
hyperbolic domain $S_f$ with $\overline{f^{-1}(S_f)}\subset S_f$, 
implying that $\Or_f$ and $\Ort_f$ have no common boundary points in $\C$. 
The same is true if the underlying surface is $\C$, 
which means that $\infty$ is the only common boundary point 
of $\Or_f$ and $\Ort_f$. Hence it only remains to check that for some $E^{'}>1$, 
\begin{eqnarray*}
\lim_{z\rightarrow\infty}\frac{\tilde{\rho}_f(z)}{\rho_f(z)}\geq E^{'}>1.
\end{eqnarray*} 
Let $C\subset\Or_f$ be the complement of a closed Euclidean disk centred at $0$ 
such that $\nu_f(z)=1$ for all $z\in C$, 
and denote by $\rho_C$ the density of the hyperbolic metric on $C$. 
Then there is a right half-plane $H\subset\C$ such that the map 
$\exp:H\rightarrow C,\;z\mapsto\e^z$ is a covering. 
Hence the asymptotic behaviour of  $\rho_C$ is given by 
\begin{eqnarray*}
\rho_C(z)= \Oo\left( \frac{1}{\vert z\vert\cdot\log\vert z\vert}\right) 
\quad\text{as}\quad z\rightarrow\infty.
\end{eqnarray*}
By Theorem \ref{pick} $\rho_C(z)\geq\rho_f(z)$, and so 
\begin{eqnarray*}
\rho_f(z)\leq\Oo \left( \frac{1}{\vert z\vert\cdot\log\vert z\vert}\right)\quad\text{as}\quad z\rightarrow\infty.
\end{eqnarray*}
It now follows from Corollary \ref{estimate_rhotilde} that 
\begin{eqnarray*}
\frac{\tilde{\rho}(z)}{\rho(z)}\geq\Oo\left(\log\vert z\vert\right) 
\end{eqnarray*}
and hence
\begin{eqnarray*}
\frac{\tilde{\rho}(z)}{\rho(z)}\rightarrow\infty\quad\text{as}\quad z\rightarrow\infty.
\end{eqnarray*}

\end{proof}

\begin{Remark}
We already mentioned in the introduction that 
if we replace the ramified points by punctures, i.e., if we consider 
the hyperbolic domain $U:\C\setminus\{ z_j\}$ instead of the orbifold $\Or$, 
the same bound for the asymptotic behaviour of the 
density map near $\infty$ can be obtained using standard estimates of 
the hyperbolic metric in the twice-punctured plane \cite[Lemma 2.1]{rempe_5}.
As will become clear in the proof of Proposition \ref{prop_cos_est},
the orbifold for which the set of ramified points
is given by $\{2 k\pi i: k\in\Z\}$  shows that our 
estimate is best possible.
\end{Remark}

\subsection{Cosine maps}
\label{subs_cosine}
We say that $F_{a,b}$ is a \emph{cosine map}, 
if it can be written as
\begin{align*}
F_{a,b}(z)=a\e^z + b\e^{-z}
\end{align*}
for some $a,b\in\C^{*}$.
Every such function $F_{a,b}$ has exactly two critical values, namely
$v_1=2\sqrt{ab}$ and $v_2=-2\sqrt{ab}$. 
Furthermore, if $z\in\C$ is a preimage of a critical value, then 
$z$ is a critical point satisfying $\deg(F_{a,b},z)=2$, 
which implies that $v_1$ and $v_2$ are totally ramified.
It is easy to check that $F_{a,b}$ has no asymptotic values, 
hence the critical values 
$v_1$ and $v_2$ are the only singular values.
In particular, every subhyperbolic cosine map 
is automatically strongly subhyperbolic.

In \cite{schleicher}, Schleicher studied landing properties of 
those cosine maps for which the critical values are strictly
preperiodic. Note that the Julia set of every such function equals $\C$. 
He proved that for such a map, every point in $\C$ is either on a dynamic ray 
or the landing point of a dynamic ray.
This result will also follow from Theorem \ref{thm_maintheorem}. 
Moreover, our proof 
in the case of strongly subhyperbolic cosine maps 
is considerably more concise and elementary than the proof of the general statement  
and the proof (of the weaker statement) given in \cite{schleicher}.
The reason is that 
for strongly subhyperbolic cosine maps, 
we can compute \emph{explicitly} the required estimates of the metrics of 
certain dynamically associated orbifolds.

Let us start with a simple observation.

\begin{Proposition}
Let $F=F_{a,b}$ be strongly subhyperbolic but not hyperbolic. 
Then there exists a point $p\in P_{\J}\setminus S(F)$.
\end{Proposition}

\begin{proof}
Since $F$ is not hyperbolic, it follows that at least one critical value of $F$,
say $v_1$,  
belongs to $\J(F)$.
Now assume that the claim is wrong, i.e., $P_{\J}=S(F)$. Since $P_{\J}$ is forward 
invariant, this can only occur if $F(v_1)=v_1$ or $v_1$ and $v_2$ form a cycle. 
However, since $v_1$ is totally ramified, it would then be a superattracting periodic 
point of $F$, contradicting the assumption that $v_1\in\J(F)$.
\end{proof}

For simplicity, let us assume that 
$\{ v_1,v_2\}\subset \J(F)$; the case when 
$\F(F)\neq\emptyset$ can be treated in a very similar way (and is even easier). 
Let $\Or_F=(\C,\nu_F)$ and $\Ort_F=(\C,\tilde{\nu}_F)$, where 
\begin{align*}
\nu_F(w)=\lcm\{\deg(F^n,z), \text{ where } F^n(z)=w\}\quad\text{and}\quad
\tilde{\nu}_F(z) = \frac{\nu(F(z))}{\deg(F,z)}.
\end{align*}
It is straightforward to check that $(\Ort_F,\Or_F)$ is a pair of orbifolds
dynamically associated to $F$. 
(In fact, this is how we constructed dynamically associated 
orbifolds in the proof of Proposition \ref{prop_Or}.)
In particular, $\nu_F(z)\in\{1,2,4\}$ for all $z\in\C$.

Let us fix a point $p\in P_{\J}\setminus S(F)$.
Then $p$ has only regular preimages $p_i$, 
for which necessarily $\tilde{\nu}_F(p_i)=\nu_F(p)\in\{ 2,4\}$.
Since $F$ is $2\pi$-periodic, 
the orbifold $\Ort_F$ is holomorphically embedded in the orbifold
$\Or_0=(\C,\nu_0)$ defined by 
\begin{eqnarray*}
\nu_0(z)=\begin{cases}\nu_F(p) &\text{ if } w=2\pi n\text{ for some }n\in\Z,\\ 1 &\text{ otherwise}.\end{cases}
\end{eqnarray*}
In particular, if $\tilde{\rho}_F$ and $\rho_0$ denote the densities 
of the hyperbolic metrics on $\Ort_F$ and$\Or_0$, respectively, then  
\begin{align*}
\tilde{\rho}_F(z)>\rho_0(z)
\end{align*}
holds for all $z\in\C$. For $\rho_0$ we can give the following explicit lower bound.
\begin{Proposition}
\label{prop_cos_est}
The density function $\rho_0(z)$ satisfies
\begin{align*}
\rho_0(z)\geq \frac{1}{C + \vert \Rea(z)\vert}
\end{align*}
for all $z\in\C$, with $4<C<6$.
\end{Proposition}

\begin{proof}
For all points in the punctured halfplane 
$\lbrace z\neq 0: \Rea(z)\leq 1/2\rbrace
\subset\C\backslash\lbrace 0,1\rbrace$, 
the density $\rho_1$ of the hyperbolic metric of $\C\setminus\{0,1\}$ 
can be bounded from below by
\begin{eqnarray*}
\frac{1}{\rho_1(z)}\leq C_1\cdot \vert z\vert\cdot(C_2 + \vert \log\vert z\vert\vert),
\end{eqnarray*}
where $C_1:=2\sqrt{2}$ and $C_2:=4 + \log(3+2\sqrt{2})$ \cite[p. 476]{beardon_pommerenke}.

Let $\Or_2:=(\C^{*},\nu_2)$, where $\nu_2(1)=2$ and $\nu_2(z)=1$ for all $z\neq 1$.
We easily see that the map 
$p:\C\backslash\lbrace 0,1\rbrace\rightarrow\Or_2,\; z\mapsto -4 (z^2 -z)$ is 
a covering map and hence a local isometry.
A simple calculation yields
\begin{eqnarray*}
\frac{1}{\rho_2(w)}\leq 2 C_1\cdot \vert\sqrt{1-w}\vert\cdot\vert 1-\sqrt{1-w}\vert\cdot(C_2 + \log 2 + \vert\log\vert 1-\sqrt{1-w}\vert\vert),
\end{eqnarray*}
where $\rho_2(z)$ is the density of the hyperbolic metric of $\Or_2$ 
and $\sqrt{z}$ denotes the principle branch of the squareroot.

Since the map $\Or_0\to\Or_2,\; z\mapsto \e^{i z}$ is a covering map,
it follows that
\begin{eqnarray*}
\frac{1}{\rho_0(z)}\leq 2 C_1\cdot \vert\sqrt{1-\e^z}\vert\cdot\vert 1-\sqrt{1-\e^z}\vert\cdot( C_2 + \log 2 + \vert\log\vert 1-\sqrt{1-\e^z}\vert\vert)\cdot\vert\e^{-z}\vert
\end{eqnarray*}
for every $z\in\Or_0$.
Let us simplify the above expression. We note, 
by expanding with $\vert 1+\sqrt{1-\e^z}\vert$, that    
\begin{eqnarray*}
\frac{\vert\sqrt{1-\e^z}\vert\cdot\vert 1-\sqrt{1-\e^z}\vert}{\vert\e^z\vert}
= \frac{\vert\sqrt{1-\e^z}\vert}{\vert 1+\sqrt{1-\e^z}\vert},
\end{eqnarray*}
and one can easily see that the obtained expression is bounded from above by $\sqrt{2}$. 

Let us now consider $\vert\log\vert 1-\sqrt{1-e^z}\vert\vert$. 
Since $\vert\log 1/z\vert=\vert\log z\vert$, it is enough to restrict to 
the case when $\vert 1-\sqrt{1-e^z}\vert\geq 1$. 
Here we get 
\begin{eqnarray*}
\vert 1-\sqrt{1-\e^z}\vert&\leq& 1+ \sqrt{\vert 1-\e^z\vert}\leq\max\lbrace 2,\; 2\sqrt{\vert 1-\e^z\vert}\rbrace\leq\max\lbrace 2,\; 2\vert \e^z\vert\rbrace
\end{eqnarray*}
and hence
\begin{eqnarray*}
\vert\log\vert 1-\sqrt{1-\e^z}\vert\vert \leq \log 2 + \vert \Rea(z)\vert.
\end{eqnarray*}
Together, these estimates yield the proof.
\end{proof}

\begin{Remark}
Let $a,b,c\in\C$ and denote by $\Or_{a,b,c}$ the 
$\C$-orbifold with signature $(2,2,2)$, with  $a$, $b$ and $c$
being the ramified points. 
Then there exists a (unique) M\"{o}bius map $M$ mapping $0$, $1$ and $-1$ to 
$a$, $b$ and $c$, respectively. Moreover, the map $z\mapsto M(\sin z)$ 
provides a covering map from $\Or_0$ to $\Or_{a,b,c}$.  
This observation enables us to estimate the hyperbolic metric of an 
arbitrary $\C$-orbifold with signature $(2,2,2)$ using simple
calculations, and hence provides an alternative way of proving 
Theorem \ref{thm_orb_main},  which --- though it is less elegant ---  
uses only elementary observations. 
\end{Remark}

\section{Construction of a semiconjugacy}
\label{sec semiconjugacy}

Recall that our goal is to construct a continuous and 
surjective map $\phi: \J(g)\rightarrow\J(f)$, 
where $g$ is any map of disjoint type that belongs to the family
\begin{eqnarray*} 
\lbrace g_{\lambda}(z)=f(\lambda z):\;\lambda\in\C\rbrace,
\end{eqnarray*}
such that
\begin{eqnarray*}
f\circ\phi(z)=\phi\circ g(z)
\end{eqnarray*}
holds for all $z\in\J(g)$. 
Recall that by \cite[Theorem 5.2]{rempe_5}, any two such maps 
$g$ and $g^{'}$ are conjugate on their Julia sets,
hence it is enough to prove the statement for 
one such map. 
We start with the construction of such a map $g$. 

Let us fix a pair of orbifolds $(\Ort_f,\Or_f)$ 
dynamically associated to $f$ 
with underlying surfaces $\widetilde{S}_f$ and $S_f$, respectively.
Note that, by Proposition \ref{prop_Or}$(c)$, 
$S_f$ can be written as $S_f=\C\setminus C$, 
where $C$ is a, possibly empty, compact set.

Observe that for every $\lambda\in\C^{*}$, $S(g_{\lambda})=S(f)$. 
Let $K>0$ be sufficiently large, such that 
$(P(f)\cup C)\subset\lbrace \vert z\vert<K/2\rbrace$. 
Since $f$ is entire, it maps bounded sets to bounded sets, 
hence there exists $L\geq K$ such that 
\begin{align*}
f^{-1}\left(\{z: \vert z\vert>L\}\right)\subset \{z: \vert z\vert>K+1\}.
\end{align*}

Let us fix a constant $L\geq K$ with this property and define $\mu:=K/L$. 
It then follows that if $g=g_{\mu}$ and $z$ is a point with $\vert g(z)\vert>L$, then
$\vert\mu z\vert>K+1$ and hence $\vert z\vert>L+L/K$. This means, 
\begin{align*}
g^{-1}\left(\{z: \vert z\vert>L\}\right)\subset \{z: \vert z\vert>L+L/K\},
\end{align*}
and, in particular, it follows from Proposition 
\ref{prop_dt} that $g$ is of disjoint type.

Define
\begin{eqnarray*}
V_j:= f^{-j}\left( \lbrace z: \vert z\vert>K\rbrace\right)\quad\text{and}\quad U_j:= g^{-j}\left( \lbrace z: \vert z\vert>L\rbrace\right).
\end{eqnarray*}

\begin{Remark}
Note that $V_j\subset\Or_f\cap\Ort_f$ holds for all $j\geq 0$, 
such as $U_{j+1}\subset U_j$, since $g$ is of disjoint type. 
Furthermore, $\J(g)$ is the set of those points that are never mapped into 
$\C\setminus U_0$, hence 
$\J(g)$ equals the limit of the domains $U_j$. 

\end{Remark}

\begin{figure}
\centering
\psfrag{Vnull}{$V_0$}
\psfrag{Veins}{$V_1$}
\psfrag{Unull}{$U_0$}
\psfrag{Ueins}{$U_1$}
\psfrag{K}{$K$}
\psfrag{L}{$L$}
\psfrag{f}{$f$}
\psfrag{g}{$g$}
\psfrag{phieins}{$\phi_1$}
\psfrag{phizwei}{$\phi_2$}
\psfrag{phieinsvonz}{$\phi_1(z)$}
\psfrag{phizweivonz}{$\phi_2(z)$}
\psfrag{gvonz}{$g(z)$}
\psfrag{z}{$z$}
\psfrag{gammaeinsvongvonz}{$\gamma_1(g(z))$}
\psfrag{gammazweivonz}{$\gamma_2(z)$}
\includegraphics[width=\linewidth]{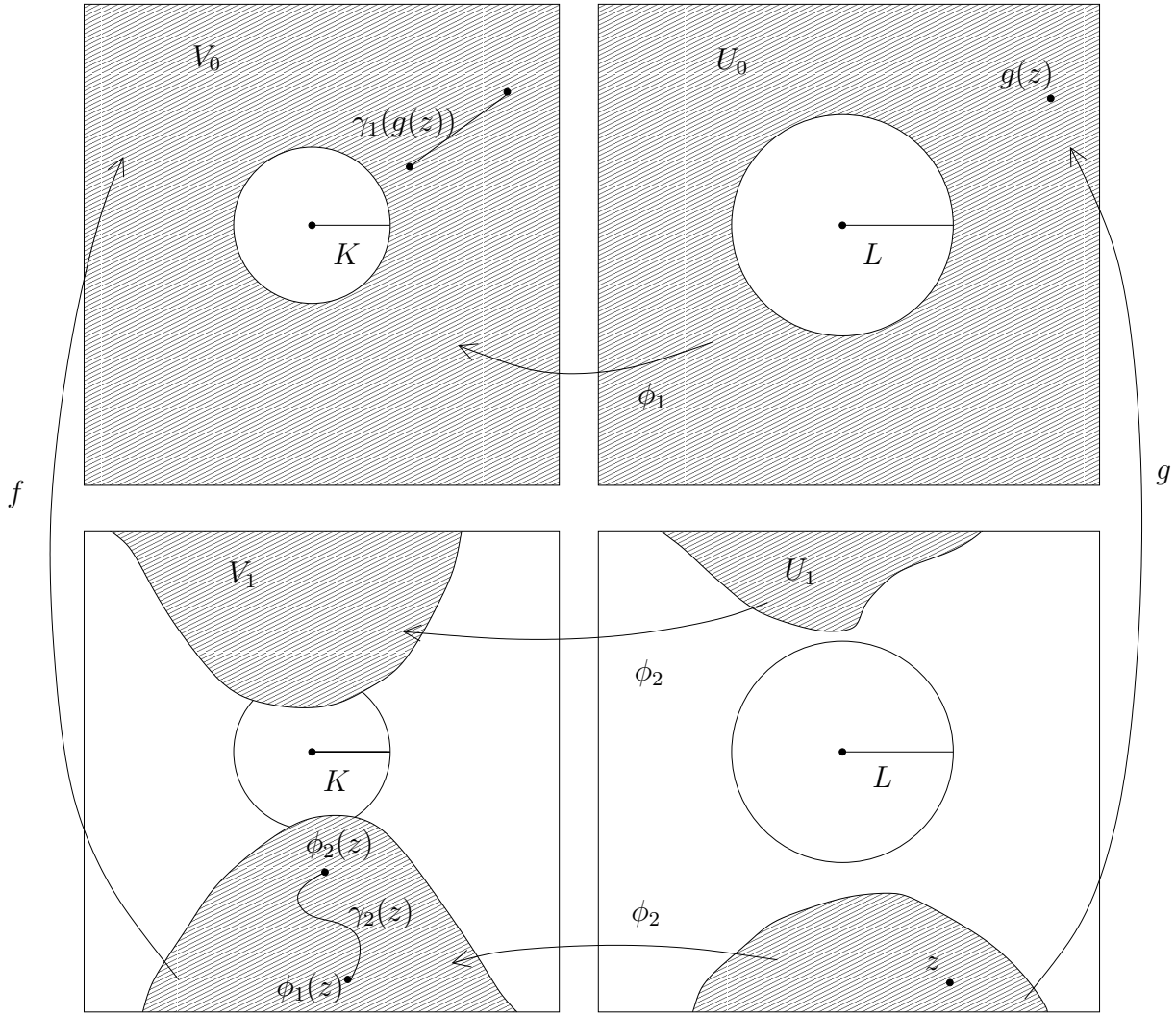}

\caption{The construction of the curve $\gamma_2(z)$ and the isomorphism $\phi_2$.}
\label{fig_semi}
\end{figure}

We want to construct a sequence of conformal isomorphisms
\begin{eqnarray*}
\phi_j:U_{j-1}\rightarrow V_{j-1}
\end{eqnarray*}
for $j\geq 1$ and with $\phi_0\equiv\id$, such that
\begin{eqnarray*}
f\circ\phi_{j+1} = \phi_j \circ g.
\end{eqnarray*}
We will proceed inductively. Since $\phi_0\equiv\id$, 
the map $\phi_1$ is given by the formula $\phi_1(z)=\mu z$. 
For a point $z\in U_0$ let $\gamma_1(z)$ be the straight line segment connecting 
$z$ and $\phi_1(z)$ (we can actually choose $\gamma_1$ to be any 
rectifiable curve which connects $z$ and $\phi_1(z)$ within the domain $V_0$). 
To define $\phi_2$ at a point $z\in U_1$, we consider the line segment 
$\gamma_1(g(z))\subset V_0$. By definition of $V_j$, 

\begin{eqnarray*}
f^{-1}\left(  \gamma_1(g(z))\right) \subset V_1.
\end{eqnarray*}

Since $f(\phi_1(z))=g(z)$, there is a preimage component 
$\gamma_2(z)$ of $\gamma_1(g(z))$, such that one endpoint 
of $\gamma_2(z)$ is $\phi_1(z)$. We define $\phi_2(z)$ 
to be the other endpoint of $\gamma_2(z)$ (see Figure \ref{fig_semi}).

Continuing inductively, we define the curve $\gamma_{j+1}(z)\subset V_j$ 
to be the pullback of $\gamma_j(g(z))\subset V_{j-1}$ under 
$f$ with one endpoint at $\phi_j(z)$, 
and we define $\phi_{j+1}(z)$ to be the other endpoint $\gamma_{j+1}(z)$.

We want to give some properties of the maps $\phi_j$. 
Since $f$ and $g$ are holomorphic and in particular continuous, 
each map $\phi_j$ is continuous as well. 
By induction, it also follows that each map $\phi_j$ is injective and surjective. 
Hence each map $\phi_j$ is a conformal isomorphism, mapping a 
component of $U_{j-1}$ onto a component of $V_{j-1}$. 

\begin{Theorem}
\label{thm_main_repeat}
The maps $\phi_j\vert_{\J(g)}$ converge uniformly with respect to the 
hyperbolic orbifold metric 
$\rho_f(z)\vert dz\vert$ on $\Or_f$ to a continuous surjective function
\begin{eqnarray*}
\phi:\J(g)\rightarrow\J(f)
\end{eqnarray*}
so that $f\circ\phi = \phi\circ g$. Moreover, $\phi(I(g))=I(f)$ and $\phi\vert_{I(g)}$ is a homeomorphism.
\end{Theorem}

\begin{proof}
With respect to the hyperbolic metric on $\Or_f$, 
we denote by $d_f(w_1, w_2)$ 
the distance between two points $w_1,w_2\in\Or_f$, 
and by $\ell_f(\gamma)$ the length of a rectifiable curve $\gamma\subset\Or_f$. 
Let  $z\in U_j$. Since $U_{j+1}\subset U_j$, 
both $\phi_j$ and $\phi_{j+1}$ are defined in a neighbourhood of $z$ and 
it follows from our construction that
\begin{eqnarray}
\label{phi_j}
d_f(\phi_{j+1}(z),\phi_j(z))\leq \ell_f(\gamma_{j+1}(z)).
\end{eqnarray}
Since for every point $z\in\Or_f$, 
\begin{eqnarray*}
\gamma_1(z)\subset\left(\C\backslash\overline{D_{K}(0)}\right) \subset\left(\C\backslash\overline{D_{\frac{K}{2}}(0)}\right) \subset\Or_f,
\end{eqnarray*}
 we obtain an upper bound for $\ell_f(\gamma_1)$ 
by computing its length with respect to the hyperbolic metric in $\C\backslash\overline{D_{\frac{K}{2}}(0)}$, which is given by 
$\left( \vert z\vert\left( \log\vert z\vert - \log (K/2)\right) 
\right) ^{-1}\vert dz\vert$. 
Hence
\begin{eqnarray*}
\ell_f(\gamma_1(z)) 
\leq \frac{1}{K} \log\left( \frac{\log\vert\lambda\vert}{\log\vert z\vert - \log (K/2)} +1\right)
\leq\frac{1}{K} \log\left(\frac{\log\vert\lambda\vert}{\log 2} +1\right) =: \len
\end{eqnarray*}
Recall that by Lemma \ref{expansion}, there is a constant $E>1$, 
such that $\Vert Df(z)\Vert_{\Or_f}\geq E$ holds for all $z\in\Ort_f$. 
Since $\gamma_{j+1}(z)\subset V_j\subset\Ort_f$ 
is obtained as a pullback of $\gamma_1(g^{j}(z))$ under the map $f^{j}$, 
it follows from equation (\ref{phi_j}) that
\begin{eqnarray*}
d_f(\phi_{j+1}(z),\phi_j(z))\leq \frac{\len}{E^{j}}.
\end{eqnarray*}
This means that the maps $\phi_j\vert_{\J(g)}$ from a Cauchy sequence, 
and since the orbifold metric is complete, there is a continuous limit function 
\begin{eqnarray*}
\phi:\J(g)\rightarrow\Or_f.
\end{eqnarray*}
Note that $\phi$ necessarily satisfies
\begin{eqnarray}
\label{eqn_escape}
d_f(\phi(z),z)\leq\sum_{j=0} ^{\infty} d_f(\phi_{j+1}(z), \phi_j(z))\leq\sum_{j=0} ^{\infty}\len\cdot\frac{1}{E^{j}} 
=\len\cdot\frac{E}{E-1}
\end{eqnarray}
as well as 
\begin{align}
\label{eqn_fe}
f^n(\phi(z)) = \phi(g^n(z))
\end{align}
for all $n\in\N$ and all $z\in\J(g)$.

We want to derive some properties of the limit function $\phi$. 
By equation (\ref{eqn_escape}), 
\begin{align}
\label{eqn_phi_infty}
\phi(z_n)\rightarrow\infty\quad\text{if and only if}\quad z_n\rightarrow\infty,
\end{align}
so together with equation (\ref{eqn_fe}) this implies that $\phi(I(g))\subset I(f)$.
In particular, it follows that $\phi(\J(g))\subset\J(f)$, since $\J(g)=\overline{I(g)}$
and $\J(f)=\overline{I(f)}$ \cite{eremenko_1}. 
Now let $w\in I(f)$. Then $w\in V_j$ 
for all sufficiently large $j$, so we can consider the sequence 
$z_j:= \phi^{-j}(w)$. Let $z$ be an accumulation point of the points $z_j$.
Note that by relation (\ref{eqn_phi_infty}), $z\neq\infty$. 
Let $z_{n_j}$ be a subsequence of $z_j$ that converges to $z$.
Then, 
\begin{eqnarray*}
\phi(z) = \phi(\lim_{n_j\rightarrow\infty} z_{n_j})=\lim_{n_j\rightarrow\infty}\phi( z_{n_j}) = w,
\end{eqnarray*}
showing that $\phi: I(g)\rightarrow I(f)$ is surjective. 

Before we show that $\phi\vert_{I(g)}$ is injective, 
let us recall that $U_1=g^{-1}\left(\{z: \vert z\vert>L\}\right)$ is a 
countable collection of simply-connected domains, so-called tracts, on which 
$g$ acts as a covering map. Now let 
$\alpha\subset U_0\setminus U_1$ be a curve connecting 
$\{z: \vert z\vert =L\}$ with $\infty$. Then every component of 
$g^{-1}(U_0\setminus\alpha)$ is a simply-connected unbounded subdomain of $U_1$, 
called a \emph{fundamental domain}, and the restriction of $g$ to any such domain $F$ is 
a conformal map. 
Now let $z,\tilde{z}\in I(g)$ be two points such 
that $\phi(z)=\phi(\tilde{z})=:w$.
By definition, $\phi_j(z),\phi_j(\tilde{z})\rightarrow w$ 
and it follows from the inductive definition of the maps $\phi_j$
that for every sufficiently large $j$, 
there exists a fundamental domain $F_j$ 
such that $g^j(z),g^j(\tilde{z})\in F_j$.
On the other hand, it follows from equation (\ref{eqn_fe})
that $\phi(g^j(z))=\phi(g^j(\tilde{z}))$ 
holds for all $j\in\N$. Furthermore, equation (\ref{eqn_escape}) implies that 
\begin{eqnarray*}
d_f(g^j(z),g^j(\tilde{z}))\leq d_f(g^j(z),\phi(g^j(z))) + 
d_f(\phi(g^j(\tilde{z})),g^j(\tilde{z}))\leq 2\len\cdot\frac{E}{E-1}.
\end{eqnarray*}
By standard expansion estimates (see e.g. \cite[Lemma 2.7]{rempe_5}),
the distance between $g^j(z)$ and $g^j(\tilde{z})$ 
must be unbounded, unless the points $z$ and $\tilde{z}$ are equal, 
implying that $\phi$ is injective.

Observe that by equation (\ref{eqn_phi_infty}), 
$\phi$ can be extended (sequentially) continuously to a map $\widehat{\phi}:\J(g)\cup\lbrace\infty\rbrace\rightarrow\J(f)\cup\lbrace\infty\rbrace$ with $\widehat{\phi}(\infty)=\infty$. 
The set $\widehat{\phi}\left(\J(g)\cup\lbrace\infty\rbrace\right)$ is compact since
it is the continuous image of a compact set.
Furthermore, $\widehat{\phi}(\J(g))=\phi(\J(g))$ 
is necessarily closed. So
\begin{eqnarray*}
I(f)=\phi(I(g))\subset\phi(\J(g))\subset\J(f)=\overline{I(f)}
\end{eqnarray*}
and as $\phi(\J(g))$ is closed, it follows that $\phi(\J(g))=\J(f)$, hence $\phi$ is surjective.
\end{proof}

Since the restriction of the map $\phi$ in Theorem \ref{thm_main_repeat}
to the escaping set of the disjoint type map is a homeomorphism,
we obtain the following result 
as an immediate consequence of Theorem \ref{thm_main_repeat} 
and Theorem \ref{thm_esc_set_dis_type}.

\begin{Corollary}
The escaping set of a strongly subhyperbolic map is not connected.
\end{Corollary}

\begin{Remark}
Dierk Schleicher kindly pointed out that the escaping set 
of the cosine map $z\mapsto\pi\sinh z$ mentioned in the introduction
is obviously disconnected: the imaginary axis consists of points with bounded
orbits and it disconnects the escaping set 
(for details on this special function, see Appendix A). 
\end{Remark}

To state our next corollary,
we need to introduce the notion of a dynamic ray.

\begin{Definition}[Dynamic rays and ray tails] 
\label{defn_ray}
A \emph{ray tail} of a transcendental entire map $f$ is an injective curve
\begin{eqnarray*}
g:[t_0,\infty)\rightarrow I(f)
\end{eqnarray*}
(where $t_0>0$) such that for each $n\in\N$, 
$\lim_{t\rightarrow\infty} f^{n}(g(t))=\infty$ and such that, as $n\to\infty$,
$f^n(g(t))\rightarrow\infty$  uniformly in $t$.

A \emph{dynamic ray} of $f$ is then a maximal injective curve 
 $g:(0,\infty)\rightarrow I(f)$ such that 
$g\vert_{[t_0,\infty)}$ is a ray tail for every $t_0>0$. 
\end{Definition}

In terms of dynamic rays, our main result implies the following 
topological description of the Julia set of certain strongly subhyperbolic maps.

\begin{Corollary}
\label{cor_pcb}
Let $f=f_1\circ\dots\circ f_n$ be a strongly subhyperbolic map, 
where every $f_i$ has finite order and a bounded set of singular values. 

Then $\J(f)$ is a pinched Cantor bouquet, consisting of 
dynamic rays of $f$ and their endpoints.
In particular, all dynamic rays of $f$ land and every point in $\J(f)$ is either
on a dynamic ray or the landing point of a dynamic ray of $f$.
\end{Corollary}

Recall that by a pinched Cantor bouquet we mean  
a quotient of a Cantor bouquet by a closed equivalence relation on its endpoints.
Note that Corollary \ref{cor_pcb} implies Corollary \ref{cor1} from the introduction.

\begin{proof}
Let $g$ and $\phi:\J(g)\to\J(f)$ be maps as in Theorem \ref{thm_main_repeat}. 
By \cite[Theorem 4.7, Theorem 5.10]{rrrs}, $\J(g)$ is an
``absorbing brush'' 
(this implies that every connected component $C$ of $\J(g)$ 
is a closed arc to infinity, and all points of $C$ except possibly
the finite endpoint escape).
In fact, $\J(g)$ is homeomorphic to a straight brush in the sense of 
\cite{aarts_oversteegen} (see Remark in \cite[p. 15]{rrrs}), 
and hence $\J(g)$ is a Cantor bouquet.

By \cite[Theorem 4.7]{rrrs}, $\J(g)$ consists 
of dynamic rays of $g$ 
and their endpoints. 
Since $\phi$ is surjective, $\J(f)$ is a pinched Cantor bouquet, 
consisting of
dynamic rays of $f$ and their endpoints.   
\end{proof}

\begin{Remark}
The statement that the absorbing brush in \cite[Theorem 4.7]{rrrs}
is homeomorphic to a straight brush in the sense of \cite{aarts_oversteegen}
can be deduced, for instance, using a topological characterization like the one 
given in \cite[Theorem 3.11]{aarts_oversteegen}.
However, we will not state the details here since it 
would require a lot vocabulary from point-set topology which would be of no 
further use in this article.
\end{Remark}

If $f$ is a function as in Corollary \ref{cor_pcb}, then 
one can encode the ``pinching'' of the Cantor bouquet (which equals $\J(f)$)
combinatorially using \emph{itineraries}. 
We will elaborate this explicitly in the case of the map 
$z\mapsto \pi\sinh z$ in the following section. 
However, such a concept can be developed 
in the more general setting of Corollary \ref{cor_pcb}; this is 
contained in the author's thesis. 
For more information on itineraries of exponential and cosine maps, 
see e.g. \cite{schleicher_zimmer,schleicher}.

\section*{Appendix A: Model of the dynamics of a map \texorpdfstring{$f$}{f} with
\texorpdfstring{$\J(f)=\C$}{J(f)=C}}

This section is dedicated to the description of the topological dynamics of the function 
\begin{align*}
f(z):=\pi\sinh z.
\end{align*}
%
We will define a ``simple'' model consisting of a topological space 
$\cl{X}$ and a map $\M:\cl{X}\to\cl{X}$ such that 
if $g$ is any map of disjoint type in the family 
$g_{\lambda}: z\mapsto\lambda\sinh z$ then 
\begin{itemize}
\item $\J(g)$ is homeomorphic to $\cl{X}$, and
\item $\M\vert_{\cl{X}}$ is conjugate to $g\vert_{\J(g)}$.
\end{itemize}

We will transfer the ideas from \cite{rempe_1}, where such a
model was constructed for exponential maps whose singular value 
belongs to some attracting basin. 
The adoption of \cite{rempe_1} to the maps we are interested in 
is particularly simple since  
in left and right half-planes, sufficiently far away from 
the imaginary axis, any map $g_{\lambda}$ with $\lambda>0$
is essentially the same (i.e., up to a constant factor)
as $z\mapsto\e^{-z}$ and $z\mapsto\e^z$, respectively.
For this reason, we will skip the details and refer, for further consideration,
to \cite{rempe_1}
as well as the extensive work on dynamics of cosine maps 
by Rottenfu\ss er and Schleicher \cite{rottenfusser_schleicher}. 

Once we have constructed such a model for a disjoint 
type map $g\in\{ g_{\lambda}\}$, Theorem \ref{thm_main_repeat}
tells us that there is a semiconjugacy between $g$ 
and $f$ on their Julia sets, and hence 
also between the model map $\M$ and $f$.
The combinatorial dynamics of $f$ on $\J(f)$ 
was already established in \cite{schleicher,schleicher_2} 
and we will summarize here the required results.

\subsection*{Dynamics within the one-parameter family}
Let us consider the family  
$g_{\lambda}(z):=\lambda\sinh z$ with $\lambda>0$ (hence $f = g_{\pi}$).
The critical values of $g_{\lambda}$ are $\pm\lambda i$. 
Every map $g_{\lambda}:\R\to\R$ is a homeomorphism with $g_{\lambda}(0)=0$
and $\R\setminus\{ 0\}\subset I(g_{\lambda})$. 
Furthermore, $g_{\lambda}(i\R)\subset [-\lambda i,\lambda i]$.

Both critical values of $f$ 
are mapped by $f$ to the repelling fixed point $0$.
Since $f$ (as well as every other $g_{\lambda}$) 
has no asymptotic values, the postsingular set 
of $f$ equals $\{\pm\pi i,0\}$. Hence $f$ is postsingularly finite 
and strongly subhyperbolic. 
Furthermore, $\J(f)=\C$.

For $\lambda>0$ chosen sufficiently small, the origin
is an attracting fixed point and the subinterval
$[-\lambda i,\lambda i]$ of the imaginary axis 
is mapped into itself and thus belongs 
to the immediate basin of attraction of $0$. 
Hence by choosing $\lambda$ sufficiently small, we obtain 
a map $g_{\lambda}$ of disjoint type (see Proposition \ref{prop_disjoint_type}).
From now on, we will fix $\lambda_0>0$ such that the corresponding map 
$g_{\lambda_0}=:g$ is of disjoint type.

Note that for every $n\in\Z$,  
the horizontal line 
\begin{align*}
L_n:=\{z: \Ima z = (n+1/2)\pi\} 
\end{align*}
is mapped by $g$ (or any other $g_{\lambda}$ with $\lambda\in\R$) 
to $i\R\setminus \left[-\lambda_0 i,\lambda_0 i\right]$, 
hence every point $z\in\J(g)$ is contained 
in one the horizontal half-strips 
\begin{eqnarray*}
S_{n_L}&:=&\{z: \Rea z<0, \Ima z\in ((n-1/2)\pi,(n+1/2)\pi)\} \text{ or } \\
S_{n_R}&:=&\{z: \Rea z>0, \Ima z\in ((n-1/2)\pi,(n+1/2)\pi)\}.
\end{eqnarray*}

Note that the restriction of $g$ 
(or any other $g_{\lambda}$ with $\lambda\in\R$)
to any of the half-strips is a conformal isomorphism onto its image
which is the left or right half-plane.

\subsection*{Topological model}

Let $\Exad:= (\Z_L\cup\Z_R)^{\N}$
be the space of infinite sequences of elements in $\Z_L\cup\Z_R$,
where $\Z_L:=\{\dots ,-1_L, 0_L, 1_L, \dots\}$ and 
$\Z_R:=\{\dots ,-1_R, 0_R, 1_R, \dots\}$ are two 
disjoint copies of $\Z$.
By the previous argument, 
we can assign to a point $z\in\J(g)$ a unique sequence 
$\ul{s}=s_0 s_1 \ldots \in\Exad$ defined by
$g^n(z)\in S_{s_n}$.
We will call such a sequence the \emph{external address} of $z$.
For every $i\in\Z$ we define $\vert i_L\vert:=\vert i\vert=:\vert i_R\vert$.
Furthermore, since $\J(g)$ consists of (asymptotically horizontal) 
dynamic rays and their endpoints \cite[Theorem 4.1]{rottenfusser_schleicher}, 
our model $\cl{X}$ should be a subset of the space 
\begin{align*}
\Exad\times [0,\infty).
\end{align*}

Note that the relation $\dots i_L<i_R<(i+1)_L<\dots$ 
defines an order on $\Exad$. 
Thus $\Exad\times [0,\infty)$ is equipped with the product topology 
of the topology on $\Exad$ (induced by the order relation) 
and the standard topology on $\R$.

Let $(\ul{s},t)$ be a point in $\Exad\times [0,\infty)$. 
We should think of the first entry 
$s_0$ in $\ul{s}$ as the imaginary part of the point 
(or its height corresponding to our horizontal strips), 
together with the information whether
it is lying left or right from the imaginary axis. 
The second entry $t$ should 
be thought of as the absolute value of the real part of the point. 
Hence it is helpful to think of a point $(\ul{s},t)\in\Exad\times [0,\infty)$
in its ``complexified'' version $C(\ul{s},t):=t+2\pi i s_0$.
Let us denote by $T(\ul{s},t):=t$ the projection onto the second coordinate.
We can now define our model map to be
\begin{align*}
\M:\Exad\times [0,\infty)\to\Exad\times [0,\infty),\; (\ul{s},t)\mapsto (\sigma(\ul{s}), F(t) - \pi\vert s_1\vert),
\end{align*}
where $\sigma$ denotes the one-sided shift map and $F(t):=\e^t -1$
denotes the standard model map for exponential growth.

Recall that the maps we consider behave like the exponential in
each of the halfplanes. 
The essential characteristic of our model map now is that as for 
exponential maps, the size of the image $\vert C(\M(\ul{s},t))\vert$ 
of a point $(\ul{s},t)$ is 
roughly the exponential of its real part. More precisely,
$F(t)/\sqrt{2} \leq \vert C(\M(\ul{s},t))\vert\leq F(t)$
whenever $T(\ul{s},t)\geq0$.
Hence we define the model sets $\cl{X}$ and $X$ to be
\begin{eqnarray*}
\cl{X}&:=&\{ (\ul{s},t)\in\Exad\times [0,\infty): T(\M^n(\ul{s},t))\geq 0\text{ for all }n\geq 0\} \text{ and }\\
X&:=&\{ (\ul{s},t)\in\cl{X}: T(\M^n(\ul{s},t))\to\infty\text{ as }n\to\infty\}.
\end{eqnarray*}

By  \cite[Observation 3.1]{rempe_1}, $\cl{X}$ is homeomorphic to a straight brush. 
In particular, for every external address
$\ul{s}$ there exists a unique $t_{\ul{s}}\in [0,\infty]$ such that 
$\{ t\geq 0: (\ul{s},t)\in \cl{X}\}=[t_{\ul{s}},\infty)$.
We denote by $E(\cl{X}):=\{(\ul{s},t_{\ul{s}})\}$ the 
set of \emph{endpoints} of $\cl{X}$.

By iterating forward under the model map $\M$ and backwards under $g$, 
we obtain a sequence of maps that converges to a homeomorphism 
$\Phi:\cl{X}\to\J(g)$ such that 
\begin{align*}
\Phi\circ \M(z)= g\circ\Phi(z)
\end{align*}
for all $z\in\cl{X}$. 
The key argument for such a limit to exist is again uniform hyperbolic contraction 
of the map $g$ and the fact that the mapping behaviour of the model map
reflects that of $g$. 
(For a precise statement see \cite[Section 3]{rempe_1} or 
\cite[Proposition 3.3]{rottenfusser_schleicher}.)
A proof of the above statement is essentially the same as in 
the case of exponential maps in \cite[Theorem 9.1]{rempe_1}, 
which is why we skip the details here.
A proof can also be derived by essentially the same estimates
as given in the proof of Theorem \ref{thm_main_repeat}.
 
By Theorem \ref{thm_main_repeat} there is a surjective map 
$\phi:\J(g)\to \J(f)$ such that $f(\phi(z))=\phi(g(z))$ holds 
for all $z\in\J(g)$. Moreover, $\phi$ restricts to a homeomorphism 
between $I(g)$ and $I(f)$.
As already mentioned, every point $z\in I(f)$ escapes 
within the strips $S_{s_i}$ with $s_i\in\Z_L\cup\Z_R$,
since the forward orbit of any point in the boundary of 
the strips has a bounded orbit. 
Recall from the proof of Theorem \ref{thm_main_repeat} that
by choosing the inverse branches of the maps $f^n$ 
appropriately,  
the conjugacy $\phi$ relates
the escaping points of $g$ and $f$ with respect to the
combinatorics in terms of their external addresses.
From Corollary \ref{cor_pcb}, we obtain that 
$\M$ projects to a function $\widetilde{\M}$ on $\widetilde{X}:=\cl{X}/\sim_p$,
where $\sim_p$ is an equivalence relation on the set 
$E(\cl{X})$ of endpoints of $\cl{X}$, 
such that $\widetilde{\M}:\widetilde{X}\to\widetilde{X}$ is 
conjugate to $f:\J(f)\to\J(f)$.
The equivalence relation $\sim_p$ tells us which dynamic rays are being ``pinched''.
We will now describe $\sim_p$
explicitly using the results from \cite{schleicher,schleicher_2}. 

\subsection*{Combinatorial description}
\begin{figure}
\centering
\includegraphics[width=\linewidth]{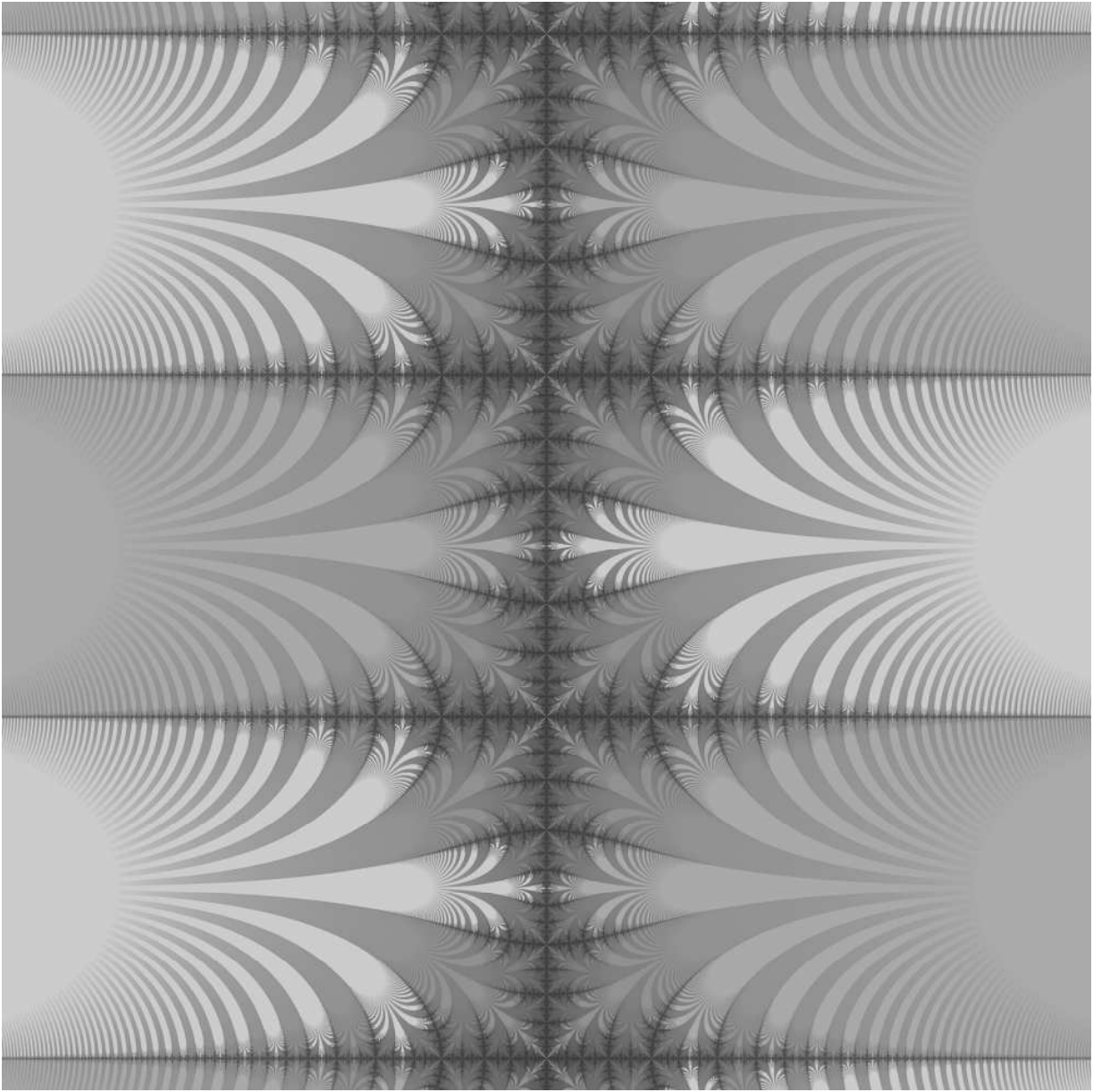}
\caption{The Julia set of the map $f(z)=\pi\sinh z$, showing its structure of a pinched 
Cantor bouquet. This picture was kindly provided by 
Arnaud Ch\'{e}ritat.}
\end{figure}

For every $n\in\Z$ we set
\begin{eqnarray*}
U_{(n,0)}&:=&\{z: \Rea z<0, \Ima z\in ((2n\pi,2(n+1)\pi)\} \text{ and } \\
U_{(n,1)}&:=&\{z: \Rea z>0, \Ima z\in (2n\pi,2(n+1)\pi)\}.
\end{eqnarray*}

One can easily see that the restrictions
$f: U_{(n,0)}\to \C\setminus (\R^{+}\cup [-\pi i,\pi i])$ and 
$f: U_{(n,1)}\to \C\setminus (\R^{-}\cup [-\pi i,\pi i])$
are conformal isomorphisms. 
We will call a sequence $(n_0,k_0) (n_1, k_1) \dots \in(\Z\times\{0,1\})^{\N}$ 
an \emph{itinerary}.
If $\gamma$ is a dynamic ray such that for every $i\geq 0$ there exists 
a domain $U_{(n_i,k_i)}$ with $f^i(\gamma)\subset U_{(n_i,k_i)}$ then
we assign to $\gamma$ the (well-defined) itinerary
$\itin(\gamma)=(n_0,k_0) (n_1,k_1)\dots$.
Since a dynamic ray of $f$ is either contained in some half-strip $U_{(n,k)}$ 
or is completely contained in the boundary of such a domain, it follows that 
an itinerary cannot be assigned to a ray $\gamma$ 
if and only if there is $n\geq 0$ such that $f^n(\gamma)$ equals $\R^{+}$
or $\R^{-}$, or equivalently, if $s_{n+j}\equiv 0_R$ or $0_L$ for all 
$j\geq 0$, where $\ul{s}=s_0 s_1 \dots$ is the external address of $\gamma$.
This means that to every external address $\ul{s}$ in 
\begin{align*}
\Stp_{+}^{\N}:=\{\ul{s}: t_{\ul{s}}<\infty\}
\setminus \{ \ul{s}:  
s_{n+j}\equiv 0_R\text{ or } 0_L \text{ for some }n\geq 0\text{ and all }j\geq 0\}
\end{align*}
we can assign a unique itinerary 
$\itin(\ul{s}):=\itin(\gamma_{\ul{s}})$.
Let us first comment on those external addresses
that belong to 
\begin{align*}
\Stp_{-}^{\N}:=\{\ul{s}: t_{\ul{s}}<\infty\}\setminus\Stp_{+}^{\N}.
\end{align*}
The mapping behaviour of the map $f$ is fairly simple and allows us 
to describe completely all tuples and quadruples of 
external addresses in $\Exad_{-}$
for which the respective dynamic rays land together.
For instance, for all addresses $\ul{s}^i$ that belong 
to either the left of right quadruple
\begin{align*}
s_0 \dots s_j\begin{cases}
\;\;\;(2m)_R \;\;\;(2n+1)_R\; 1_L\; \cl{0_R}\\
\;\;\;(2m)_R \;\;\;\;\;\;(2n)_R \;\;\;\;1_R\; \cl{0_L}\\
(2m+1)_R (2n+1)_L\; 1_R\; \cl{0_L}\\
(2m+1)_R\;\;\; (2n))_L\;\;\; 1_L\; \cl{0_R}
             \end{cases} 
\qquad
 s_0 \dots s_j\begin{cases}
(2m+1)_L(2n+1)_R\; 1_L\; \cl{0_R}\\
(2m+1)_L \;\;\;(2n)_R \;\;\;\;1_R\; \cl{0_L}\\
\;\;\;(2m)_L \;\;\; (2n+1)_L\; 1_R\; \cl{0_L}\\
\;\;\;(2m)_L\;\;\;\;\;\; (2n))_L\;\;\; 1_L\; \cl{0_R}
             \end{cases} 
\end{align*}
where $m\in\Z$ and $n\geq 0$, we define 
$(\ul{s}^i,t_{\ul{s}^i})\sim_p (\ul{s}^j,t_{\ul{s}^j})$. 
We will not list the remaining 
combinations since there are not so many of them and each one  
is easy to determine using elementary computations.

The remaining task is to determine those
external addresses in $\Stp_{+}^{\N}$ such that the corresponding dynamic rays
land together. Let $\gamma$ be a dynamic ray with external address
$\ul{s}\in\Exad_{+}$ and let $w$ be its landing point.  
Then either $g$ is the only dynamic ray that lands at $w$ 
or there is exactly one more such dynamic ray \cite{schleicher_2};
the latter case occurs if and only if $w$ is eventually mapped 
into $[\pi i,-\pi i]$ (and remains there without ever being mapped to $0$). 
Let $\ul{s},\tilde{\ul{s}}$ be the external addresses of two dynamic rays landing 
at the same point $w$. 
It follows that 
$\itin(\ul{s})$ and $\itin(\tilde{\ul{s}})$ must be of the form 
\begin{align}
\label{eqn_itin}
(n_0,k_0) \dots (n_j,k_j)\begin{cases}
(n_{j+1},\;\;k_{j+1}\;\;) (0,\;\;k_{j+2}\;) (0,\;\;k_{j+3}\;\;)\dots\\
(n_{j+1},\!1\!\!-\!k_{j+1}) (0,\!1\!\!-\!k_{j+2}) (0,\!1\!\!-\!k_{j+3})\dots
                         \end{cases}
\end{align}
or with $-1$ instead of $0$.


On the other hand, it follows from \cite[Lemma 5]{schleicher} and 
elementary computations that two dynamic rays with itineraries 
as in equation (\ref{eqn_itin}) do indeed land together: 
such dynamic rays have a forward image that lands in the interval 
$[-\pi i,\pi i]$ and its landing point is never mapped to $0$.
So let $\ul{s},\tilde{\ul{s}}\in\Stp_{+}^{\N}$. It follows that 
$(\ul{s},t_{\ul{s}})\sim_p (\tilde{\ul{s}},t_{\tilde{\ul{s}}})$ if and only if
$\itin(\ul{s})$ and $\itin(\tilde{\ul{s}})$ are of the form given by 
equation (\ref{eqn_itin}) (or with $-1$ instead of $0$).

\begin{Remark}
One can certainly relate the model $(\cl{X},\M)$ directly to  
$\J(f)$. The reason to incorporate 
a disjoint type map is simply 
to show what to do when the considered strongly subhyperbolic map $f$ can be embedded 
in a family where the topological dynamics of disjoint type maps is well understood.
\end{Remark}
\section*{Appendix B: Subhyperbolic maps without dynamically associated orbifolds}

We want to give an explicit example of a map $\Phi$ that is subhyperbolic, 
has no asymptotic values but such that  
the local degree at points in $\J(\Phi)$ 
is unbounded. 

Let $p$ be a complex polynomial of degree $d\geq 2$ and 
let $z_0$ be a repelling fixed point of $p$ with multiplier $\mu$. 
By Poincar\'{e}'s Theorem \cite{poincare,valiron}, there exists an entire map $\Phi$, 
which is called a \emph{Poincar\'{e} function of $p$ at $z_0$}, 
such that the functional equation 
\begin{align}
\label{eqn_poincare}
\Phi(\mu\cdot z)=p(\Phi(z))
\end{align}
is satisfied for all $z\in\C$. 
Now let 
\begin{align*}
p(z)=z^2-1
\end{align*}
and let $\Phi_0$  denote a Poincar\'{e} function of $p$ at 
the point $z_0=1/2(1+\sqrt{5})$.
The unique finite critical point of $p$ is $0$. 
Since $p(0)=-1$ and $p(-1)=0$, the cycle $\{0,-1\}$ is superattracting
and, in particular, $P(p)\cap \C=\{0,-1\}$. 
Note that $p$ has no other attracting or parabolic cycles in $\C$, since 
every such cycle attracts at least one critical point of $p$ \cite[Theorem 8.6]{milnor}.
Observe also that $p$ has no exceptional values 
(points with a finite backward orbit) in $\C$; it is now not hard to check that 
$C(\Phi_0)=P(p)=\{0,-1\}$ (see e.g. \cite[Theorem 2.10]{drasin_okuyama}).

The multiplier of the repelling fixed point $z_0$ is given by 
$\lambda=p^{'}(z_0) = 1+\sqrt{5}$.
By \cite[p. 160]{valiron}, the order of $\Phi_0$ is given by the formula
\begin{align*}
\rho(\Phi_0)=\frac{\log 2}{\log\vert\lambda\vert}<\frac{\log 2}{\log 3}<1
\end{align*}
and it follows then from the Denjoy-Carleman-Ahlfors Theorem 
\cite[XI, $\S$4, p.313]{nevanlinna} that $\Phi_0$ has 
at most one finite asymptotic value.
Let us assume that $A(\Phi_0)\neq\emptyset$ and let $w$ be the unique 
asymptotic value of $\Phi_0$. 
By \cite[Theorem 1]{drasin_okuyama}, $w$ 
is an attracting periodic point of $p$, hence either $w=0$ or $w=-1$. 
Let $\gamma(t)$ be an asymptotic path for $w$, i.e., 
$\lim_{t\to\infty} \gamma(t)=\infty$ 
and $\lim_{t\to\infty} \Phi_0(\gamma(t)) =w$. Then 
\begin{align*}
 \lim_{t\to\infty} \Phi_0(\lambda\gamma(t))= p(\lim_{t\to\infty}\Phi_0(\gamma(t))) = p(w)\neq w,
\end{align*}
hence $\gamma_{\lambda}(t):=\lambda\cdot\gamma(t)$ is an asymptotic path of $\Phi_0$ leading to 
the asymptotic value $p(w)$. But this contradicts the fact that $A(\Phi_0)=\{ w\}$,
and hence $A(\Phi_0)=\emptyset$.

Since $0$ is a critical value of $\Phi_0$, there exists a point $\tilde{z}$ such that
\begin{align*}
 \Phi_0(\tilde{z})=0\quad\text{and}\quad\Phi_0^{'}(\tilde{z})=0.
\end{align*}
 Let $z_n:=\lambda^n\tilde{z}$.
Using the functional equation (\ref{eqn_poincare}), it is not difficult to see 
that  for every $n\in\N$, 
\begin{align*}
 \frac{d^k}{dz^k}\Phi_0(z)\vert_{z=z_n} =0 \quad\text{for all } 0\leq k\leq n,
\end{align*}
hence for every $n\in\N$, we have $\deg(\Phi_0,z_n)\geq n$. 
This also implies that at least one of the critical values must have a regular preimage;
otherwise, both critical values would have ramification index 
strictly larger than $1/2$ and this is not possible by the fact 
that the sum of the ramification indices of an entire map never 
exceeds one \cite[Chapter X, $\S$3, 235]{nevanlinna}.
We can assume w.l.o.g. that $0$ has a regular preimage, 
say $a$, i.e., $\Phi_0(a)=0$ and $\Phi_0^{'}(a)\neq 0$.
Let $b$ be a preimage of $-1$, chosen sufficiently large 
such that $\vert a-b\vert\cdot\vert \Phi_0^{'}(a)\vert>1$. Now consider the map 
\begin{align*}
 \Phi(z):=(a-b)\cdot\Phi_0(z) +a.
\end{align*}
Note that $A(\Phi)=\emptyset$, since $\Phi$ and $\Phi_0$ differ only by 
postcomposition with a conformal map. 
It follows also immediately that $C(\Phi)=\{ a,b\}$,   
$\Phi(a)=a$ and $\Phi(b)=b$, hence 
$\Phi$ is postsingularly finite and in particular subhyperbolic. 
Moreover, since 
$\vert\Phi^{'}(a)\vert=\vert a-b\vert\cdot\vert\Phi_0^{'}(a)\vert>1$, the 
critical value $a$ is a repelling fixed point of $\Phi$ and hence belongs 
to $\J(\Phi)$, implying that $\Phi$ is not hyperbolic. 
Finally note that the points $z_n$ are mapped to $a$ under $\Phi$ satisfying  
$\deg(\Phi,z_n)\geq n$, so $\Phi$ is not strongly subhyperbolic.

Altogether, this means that $\Phi(z)$ is subhyperbolic, $A(\Phi)=\emptyset$ but 
for every $n\in\N$ there exists a point $z_n\in\J(\Phi)$ such that $\deg(\Phi,z_n)\geq n$,
yielding the desired example.

\begin{Remark}
It is not hard to see that one can use the same idea to construct many more 
Poincar\'{e} maps with the desired properties. 
\end{Remark}

\end{document}